\documentclass[a4paper,10pt,twoside,reqno]{amsart}

\usepackage[utf8]{inputenc}
\usepackage{amsmath, amsfonts, amssymb,amsthm}
\usepackage[mathscr]{eucal}
\usepackage[pdftex]{graphicx}
\usepackage{color}
\usepackage{multicol}

\usepackage[T1]{fontenc}
\usepackage[sc, osf]{mathpazo}
\usepackage{enumitem}
\setlist[itemize]{topsep=0pt, leftmargin=2em}
\setlist[enumerate]{topsep=0pt, leftmargin=2em}

\usepackage[font=small]{caption}

\newcommand{\N}{\mathbb{N}}

\newcommand{\R}{\mathbb{R}}
\newcommand{\C}{\mathbb{C}}
\newcommand{\s}{\mathbb{S}}
\newcommand{\h}{\mathbb{H}}
\newcommand{\E}{\mathbb{E}}

\newcommand{\df}{\mathrm{d}}
\newcommand{\prodesc}[2]{\left\langle#1,#2\right\rangle}
\newcommand{\abs}[1]{\left\lvert#1\right\rvert}
\DeclareMathOperator{\arctanh}{arctanh}
\DeclareMathOperator{\sech}{sech}
\DeclareMathOperator{\pIm}{Im}

\DeclareMathOperator{\Area}{Area}

\newtheorem{theorem}{Theorem}[section]
\newtheorem{proposition}[theorem]{Proposition}
\newtheorem{corollary}[theorem]{Corollary}
\newtheorem{lemma}[theorem]{Lemma}

\theoremstyle{definition}

\theoremstyle{remark}
  \newtheorem{remark}[theorem]{Remark}

\numberwithin{equation}{section}

\title[Horizontal Delaunay surfaces in $\s^2\times\R$ and $\h^2\times \R$]{Horizontal Delaunay surfaces with constant mean curvature in $\s^2\times\R$ and $\h^2\times \R$}

\author{José M. Manzano}
\address{Departamento de Matemáticas, Universidad de Jaén, Spain.
}
\email{manzanopreg@gmail.com}

\author{Francisco Torralbo}
\address{Departamento de Geometría y Topología, Universidad de Granada, Spain}
\email{ftorralbo@ugr.es}

\subjclass[2010]{Primary 53A10; Secondary 53C30}

\keywords{Constant mean curvature surfaces, compact surfaces, homogeneous three-manifolds, product spaces, conjugate constructions, unduloid, nodoid.}

\begin{document}

\begin{abstract}
  We obtain a $1$-parameter family of \emph{horizontal Delaunay surfaces} with positive constant mean curvature in $\mathbb{S}^2\times\mathbb{R}$ and $\mathbb{H}^2\times\mathbb{R}$, being the mean curvature larger than $\frac{1}{2}$ in the latter case. These surfaces are not equivariant but singly periodic, lie at bounded distance from a horizontal geodesic, and complete the family of horizontal unduloids given in~\cite{MT}. We study in detail the geometry of the whole family and show that horizontal unduloids are properly embedded in $\mathbb H^2\times\mathbb{R}$. We also find (among unduloids) families of embedded constant mean curvature tori in $\mathbb S^2\times\mathbb{R}$ which are continuous deformations from a stack of tangent spheres to a horizontal invariant cylinder. In particular, we find the first non-equivariant examples of embedded tori in $\mathbb{S}^2\times\mathbb{R}$, which have constant mean curvature $H>\frac12$. Finally, we prove that there are no properly immersed surface with constant mean curvature $H\leq\frac{1}{2}$ at bounded distance from a horizontal geodesic in $\mathbb{H}^2\times\mathbb{R}$.
\end{abstract}

\maketitle

\section{Introduction}

In 1841, Delaunay classified the $H$-surfaces in Euclidean space $\R^3$ with rotational symmetry for all $H>0$, where the prefix $H$ indicates the surface has constant mean curvature $H$. Rotationally invariant $H$-surfaces are well known in Riemannian homogeneous simply connected three-manifolds admitting a $1$-parame\-ter group of rotations, namely in space forms $\mathbb{M}^3(c)$ of constant sectional curvature $c\in\R$ or in the so-called $\mathbb{E}(\kappa, \tau)$-spaces, $\kappa-4\tau^2\neq 0$, whose isometry group has dimension four, see~\cite{HH89,PR99,Tomter1993, Torralbo2010}. Delaunay-type surfaces show up in $\mathbb{M}^3(c)$ when $|H|>-c$ and in $\mathbb{E}(\kappa, \tau)$ when $4H^2+\kappa>0$, and they usually display similar geometric shapes despite of the variety of the aforesaid ambient spaces; namely, they are essentially spheres, cylinders, unduloids and nodoids. In $\mathbb{E}(\kappa,\tau)$, the value of $H$ (if any) such that $4H^2+\kappa=0$ is usually called \emph{critical mean curvature} in the literature, because surfaces with subcritical, critical and supercritical mean curvature usually display very different geometric features, see~\cite{Daniel07,Man12,Mazet2015,MPR} and the references therein.

In space forms $\mathbb{M}^3(c)$ with $c\leq 0$, the only properly embedded $H$-surfaces with $|H|>-c$ that stay at a bounded distance from a geodesic are Delaunay surfaces. This condition is often called \emph{cylindrical boundedness} and can be relaxed to the topological assumptions of finite genus and two ends, see~\cite{KKS1989, KKMS1992}. On the contrary, cylindrical boundedness makes little sense in the three-sphere $\mathbb S^3$, and proper embeddedness with two ends is naturally replaced by assuming the surface is an embedded torus. Andrews and Li~\cite{AL2015}, building upon the work of Brendle~\cite{Brendle2013}, characterized the embedded $H$-tori in $\mathbb S^3$ as Delaunay surfaces invariant by rotations about a geodesic. Embeddedness plays an essential role since there exist immersed non-rotational $H$-tori constructed by Bobenko~\cite{Bobenko1991}.

In the case of the product spaces $\mathbb{M}^2(\kappa)\times\mathbb{R}=\mathbb{E}(\kappa,0)$, where $\mathbb{M}^2(\kappa)$ stands for the complete simply connected surface of constant curvature $\kappa$, Mazet~\cite{Mazet2015} characterized unduloids in $\mathbb{H}^2(\kappa)\times\mathbb{R}$ as the only properly embedded finite-topology $H$-surfaces which are cylindrically bounded with respect to a vertical geodesic. His result also applies to the product of an hemisphere of $\s^2(\kappa)$ and the real line, though it cannot be extended to the whole $\mathbb{S}^2(\kappa)\times \mathbb{R}$. This is a consequence of the fact that there do exist compact non-rotational $H$-surfaces in $\mathbb{S}^2(\kappa)\times\mathbb{R}$, constructed by the authors in~\cite{MT}, which are now proved embedded by Theorem~\ref{thm:embeddedness} below. We have called such examples \emph{horizontal unduloids}, and they exist in both $\mathbb{S}^2(\kappa)\times\mathbb{R}$ and $\mathbb{H}^2(\kappa)\times\mathbb{R}$ provided that $4H^2+\kappa>0$. The naming is motivated by their invariance under a discrete group of horizontal translations as well as by the fact that their shapes resemble those of Delaunay's unduloids. They are cylindrically bounded with respect to a horizontal geodesic and, for a fixed value of $H$, form a continuous $1$-parameter family of $H$-surfaces from a stack of rotationally invariant $H$-spheres to an $H$-cylinder ($H$-torus if $\kappa > 0$) invariant under a continuous group of horizontal translations. 

In this paper we incorporate a $1$-parameter family of \emph{horizontal nodoids} which completes the aforementioned family of unduloids in both $\mathbb{S}^2(\kappa)\times\mathbb{R}$ and $\mathbb{H}^2(\kappa)\times\mathbb{R}$. This is tackled by considering a Plateau problem over an appropriate geodesic polygon in a three-manifold (locally isometric to a Berger sphere), whose solution is conjugate ---in the sense of Daniel~\cite{Daniel07}--- to a fundamental piece of the desired nodoid, and then it is completed by successive mirror symmetries of the ambient product space. It is important to mention that the new polygons we have developed are not Nitsche graphs in the vertical direction (they have two horizontal components projecting onto the same geodesic of $\mathbb{S}^2(\kappa)$ via the Hopf fibration), and hence this is the first conjugate construction in product spaces whose fundamental piece is not a vertical graph. This is an additional difficulty, since most arguments developed in literature strongly depend upon the graphical condition, and also because the Plateau problem is not well posed to apply Meeks and Yau's solution~\cite{MY82} in a Berger sphere (as in the case of unduloids) but in the universal cover of some subset. We will come up with a new approach based on the comparison with subsets of Clifford tori to understand the interior points with vertical tangent plane, as well as finding a Killing direction in which the surface is really a graph (this will be discussed in Section~\ref{sec:geometry}). We will obtain a faithful depiction of the new surfaces, in particular showing that they actually look like Delaunay's nodoids (see Figure~\ref{fig:conjugate-polygon}). The following statement summarises the whole family of horizontal Delaunay $H$-surfaces.

\begin{theorem}\label{thm:nodoids}
  Fix $\kappa\in\R$ and a horizontal geodesic $\Gamma\subset\mathbb{M}^2(\kappa)\times\{0\}$. There exists a family $\Sigma_{\lambda,H}^*$, parametrized by $\lambda\geq 0$ and $H>0$ such that $4H^2+\kappa>0$, of complete $H$-surfaces in $\mathbb{M}^2(\kappa)\times \R$, invariant under a discrete group of translations along $\Gamma$ with respect to which they are cylindrically bounded. They are also symmetric about the totally geodesic surfaces $\mathbb{M}^2(\kappa)\times\{0\}$ and $\Gamma\times\mathbb R$. Moreover:
\begin{enumerate}[label=(\roman*)]
  \item $\Sigma_{0,H}^*$ is the $H$-cylinder ($H$-torus if $\kappa>0$) invariant under the continuous $1$-parameter group of translations along $\Gamma$;
  \item $\Sigma_{\lambda,H}^*$ is the unduloid-type surface constructed in~\cite{MT} if $0<\lambda<\frac{\pi}{2}$;
  \item $\Sigma_{\frac{\pi}{2},H}^*$ is a stack of tangent rotationally invariant $H$-spheres centered on $\Gamma$;
  \item $\Sigma_{\lambda,H}^*$ is a nodoid-type surface if $\lambda>\frac{\pi}{2}$. 
\end{enumerate}
\end{theorem}

Although $\kappa$ may be assumed equal to $-1$, $0$ or $1$ after scaling the metric, we would rather keep it as a real number to understand how the case $\kappa=0$ fits in the whole family. Observe that, if $\kappa=0$, the surface $\Sigma_{\lambda,H}^*$ is one of the classical Delaunay $H$-surfaces in $\R^3$ (see Remarks~\ref{rmk:round-sphere1} and~\ref{rmk:round-sphere2}), in which case the parameter $H$ represents a variation by homotheties once $\lambda$ is fixed. Their conjugate minimal surfaces in the round sphere $\mathbb{S}^3(H^2)$ are the so-called \emph{spherical helicoids} (see Section~\ref{subsec:conjugate-Plateau-technique} and also~\cite[Proposition~1]{MT}).

The authors~\cite{MTAJM} have recently constructed the first examples of compact embedded $H$-surfaces in $\mathbb{S}^2(\kappa)\times\R$ with genus $g\geq 2$ and $H<\frac{1}{2}$, and the existence of non-equivariant examples in the cases $g=1$ or $H\geq\frac{1}{2}$ remained unknown. Embeddedness is usually tough in conjugate constructions, specially when no Krust-type property holds true. In~\cite{MTAJM}, embeddedness was achieved by proving the convexity of the boundary of the domain of $\mathbb{S}^2(\kappa)$ over which the compact surface is a bigraph using the estimates in~\cite{Man12}. As for the present surfaces $\Sigma_{\lambda,H}^*$, we tackle embeddedness by identifying Killing vector fields in Berger spheres and in $\mathbb{M}^2(\kappa)\times\R$ that produce the same function in the kernel of the common stability operator of conjugate surfaces. Irrespective of $\kappa\in\R$, our proof goes through proving that the fundamental annulus is a maximal stable domain of $\Sigma_{\lambda,H}^*$, and from there we infer that it is a graph with respect to a horizontal direction (see Proposition~\ref{prop:stability} and Figure~\ref{fig:conjugate-polygon}). If $\kappa\leq 0$, this establishes that horizontal unduloids are properly embedded (Proposition~\ref{prop:properly-embedded-unduloids}), as conjectured in~\cite{MT}. If $\kappa>0$, among all surfaces given by Theorem~\ref{thm:nodoids}, next result determines which ones are compact and embedded, whose moduli space is represented in Figure~\ref{fig:compact-embedded-moduli-space}. Note that horizontal nodoids are not even Alexandrov-embedded for any $\kappa\in\R$.

\begin{theorem}\label{thm:embeddedness}
Fix $\kappa>0$. For each integer $m\geq 2$, there is a family $\mathcal T_m$ of embedded $H$-tori in $\mathbb{S}^2(\kappa)\times\mathbb R$ parametrized as
\[\mathcal T_m=\left\{\Sigma_{\lambda_m(H),H}^*:\cot(\tfrac{\pi}{2m})<\tfrac{2H}{\sqrt{\kappa}}\leq\sqrt{m^2-1}\right\}.\]
where $H\mapsto\lambda_m(H)$ is a continuous strictly decreasing function ranging from $\frac{\pi}{2}$ to $0$.
\begin{enumerate}
  \item The family $\mathcal T_m$ is a continuous deformation (in which $H$ varies) from a stack of $m$ tangent spheres evenly distributed along $\Gamma$ to an equivariant torus.
  \item The surfaces $\Sigma_{\lambda_m(H),H}^*$, along with $H$-spheres $\Sigma_{\pi/2,H}^*$ and $H$-cylinders $\Sigma_{0,H}^*$ for all $H>0$, are the only compact embedded $H$-surfaces among all $\Sigma_{\lambda,H}^*$ (for all $\kappa\in\R$).
\end{enumerate}
\end{theorem}

\begin{figure}[htb]
\includegraphics{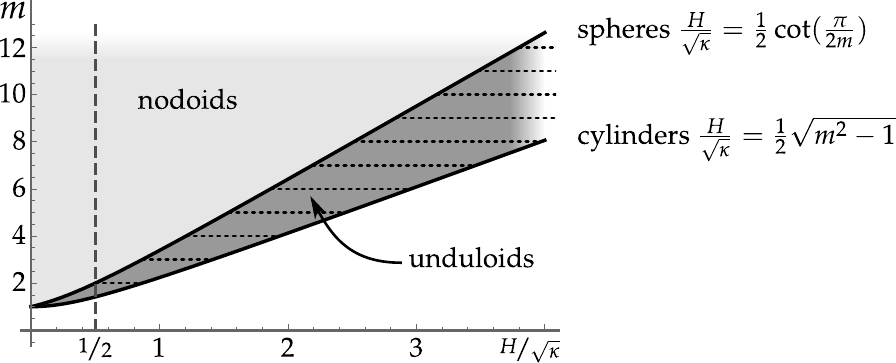}
\caption{The darker shaded region represents the moduli space of $\Sigma_\lambda^*\subset\mathbb{S}^2(\kappa)\times\mathbb{R}$, $\lambda\in[0,\frac\pi2]$, in terms of $\frac{H}{\sqrt\kappa}$ and $m$. Dotted horizontal segments indicate compact embedded unduloids as solutions to the inequality~\eqref{thm:embeddedness:eqn2} with integer $m$. The vertical dashed line indicates that such embedded unduloids exist if and only if $H>\frac{\sqrt\kappa}{2}$.}\label{fig:compact-embedded-moduli-space} 
\end{figure}

In particular, given $\kappa>0$ and $H>\frac{\sqrt{\kappa}}{2}$, there are (finitely many) compact embedded $H$-unduloids in $\mathbb{S}^2(\kappa)\times\mathbb{R}$, and these are the first known embedded $H$-tori in $\mathbb{S}^2(\kappa)\times\mathbb{R}$ (for any $H$) which are not equivariant. The case $H=\frac{\sqrt{\kappa}}{2}$ occurs as a limit surface for $m=2$ and consists of two $\frac{\sqrt{\kappa}}{2}$-spheres tangent along a common equator, each of which is a bigraph over an hemisphere. As in~\cite{MTAJM}, we find again an obstruction at $H=\frac{\sqrt{\kappa}}{2}$, which gives additional evidence that this value is important for the existence of compact embedded $H$-surfaces in $\mathbb{S}^2(\kappa)\times\mathbb{R}$. It is fundamental to remark that this value of $H$ is not related to the aforementioned notion of critical mean curvature. Note also that compact examples are dense in the family $\Sigma_{\lambda,H}^*$, showing up just when $\lambda$ satisfies a rationality condition (see Remark~\ref{rmk:compactness}), though they are never embedded if $H\leq\frac{\sqrt{\kappa}}{2}$.

We are also interested in the maximum height that $\Sigma_{\lambda,H}^*$ reaches over the horizontal slice of symmetry. We will show that the maximum height of $\Sigma^*_{\lambda, H}$ is strictly increasing in the parameter $\lambda$ (see Proposition~\ref{prop:height}). In particular, the height of a horizontal unduloid is strictly between the heights of the sphere and the cylinder. Also, horizontal nodoids are taller than the corresponding $H$-spheres, so we can confirm that the Serrin-type height estimates in~\cite{AEG}, as well as the boundary curvature estimates in~\cite{Manzano}, fail in general for symmetric surfaces which are not bigraphs even though their heights might be bounded.

It is important to point out why the condition $4H^2+\kappa>0$ appears naturally in Theorem~\ref{thm:nodoids}. The many dissimilarities between supercritical, critical and subcritical $H$-surfaces can be explained by the fact that their conjugate minimal surfaces belong to Berger spheres $\mathbb{S}^3_b(4H^2+\kappa, H)$ if $4H^2+\kappa>0$, the Heisenberg space $\mathrm{Nil}_3 = \mathbb{E}(0,H)$ if $4H^2+\kappa=0$, or the universal cover of the special linear group $\widetilde{\mathrm{SL}}_2(\R)=\mathbb{E}(4H^2+\kappa, H)$ if $4H^2+\kappa<0$, whose geometries are really different. As a matter of fact, the required geodesic polygons in our construction do not even exist in $\mathrm{Nil}_3$ or $\widetilde{\mathrm{SL}}_2(\R)$. This made us surmise the nonexistence of $H$-surfaces cylindrically bounded with respect to a horizontal geodesic if $4H^2+\kappa\leq0$.

\begin{theorem}\label{thm:halfspace}
There exist properly immersed $H$-surfaces in $\mathbb{M}^2(\kappa)\times\mathbb{R}$ cylindrically bounded with respect to a horizontal geodesic if and only if $4H^2+\kappa>0$.
\end{theorem}

Existence in Theorem~\ref{thm:halfspace} is guaranteed by Theorem~\ref{thm:nodoids}, and nonexistence is a consequence of the fact that the family of equivariant $H$-cylinders $\Sigma_{0,H}^*\subset\h^2(\kappa)\times\R$ foliates $\mathbb{H}^2(\kappa)\times\mathbb{R}$ minus a horizontal geodesic when $H$ ranges from $\frac{\sqrt{-\kappa}}{2}$ to $+\infty$ (Lemma~\ref{lemma:foliation}). This key property enables the application of Mazet's halfspace theorem for parabolic $H$-surfaces~\cite{Mazet2013}. We would like to remark that cylindrical boundedness seems to be a sharp assumption in Theorem~\ref{thm:halfspace}: on the one hand, there do exist properly immersed $H$-surfaces in $\mathbb{M}^2(\kappa)\times\mathbb{R}$ with $4H^2+\kappa\leq 0$ lying in a slab between two horizontal slices, see~\cite{MTAJM}; on the other hand, there are $H$-surfaces in $\mathbb{H}^2(\kappa)\times\mathbb{R}$ with $4H^2+\kappa<0$ at bounded distance from a totally geodesic vertical plane, e.g., the equidistant vertical planes. Existence of such a surface in the critical case $4H^2+\kappa=0$ is not hitherto known.

\medskip

\noindent\textbf{Acknowledgement.} 
The authors are supported by the Spanish \textsc{micein} project \textsc{pid}2019-111531\textsc{ga}-\textsc{i}00 and \textsc{mineco} project \textsc{mtm}2017-89677-\textsc{p}. The first author is also supported by the University of Jaén research program \textsc{acción 10}; the second author is also supported by the Programa Operativo \textsc{feder} Andalucía 2014-2020, grant no.\ \textsc{e-fqm-309-ugr18}. The authors would like to thank L.\ Mazet for some valuable comments concerning the application of the halfspace theorem in~\cite{Mazet2013}.

\section{A foliation by horizontal $H$-cylinders}\label{sec:foliation}

Consider the $1$-parameter group of translations $\{\Phi_t\}_{t\in\mathbb{R}}$ along a given a horizontal geodesic $\Gamma\subset\mathbb{M}^2(\kappa)\times\mathbb{R}$, i.e., the $\Phi_t$ are hyperbolic translations if $\kappa<0$, Euclidean translations in $\kappa=0$, or rotations if $\kappa>0$. If $4H^2+\kappa>0$, there is a unique $H$-cylinder $C_H$ invariant under the action of $\{\Phi_t\}_{t\in\mathbb{R}}$, see~\cite{HH89,PR99} and also~\cite{Man12}. This is the surface $\Sigma_{0,H}^*$ that appears in Theorem~\ref{thm:nodoids}, but at this moment we are interested in the fact that $\{C_H:\frac{\sqrt{-\kappa}}{2}<H<+\infty\}$ produces a foliation when $\kappa<0$. This property is evident if $\kappa=0$ but fails if $\kappa>0$, see~\cite[Figure~2]{Manzano}. We will assume $\kappa=-1$ in the sequel after scaling the metric.

Consider the halfspace model of $\mathbb{H}^2\times\mathbb{R}$ given by $\{(x,y,z)\in\R^3:y>0\}$ endowed with the Riemannian product metric $y^{-2}(\df x^2+\df y^2)+\df z^2$. In this model, we can assume that $\Gamma=\{(0,y,0):y>0\}$ and $\Phi_t(x,y,z)=(e^tx,e^ty,z)$. The surface $P$ given by $x^2+y^2=1$ is a totally geodesic flat vertical plane, which can be parametrized isometrically as $(r,h)\mapsto(\tanh(r),\sech(r),h)$, where $r$ is the hyperbolic distance to $(1,0)$ in $\mathbb{H} ^2$ and $h$ is the projection onto the factor $\R$. Therefore, a regular surface invariant under $\Phi_t$ can be parametrized as
\begin{equation}\label{eqn:invariant-parametrization}
\phi(t,u)=\left(e^t\tanh(r(u)),e^t\sech(r(u)),h(u)\right),
\end{equation}
for some regular curve $\alpha_H(u) = (r(u), h(u))$ in the Euclidean $(r,h)$-plane. Given $H>\tfrac{1}{2}$, by either checking the corresponding \textsc{ode}, or by intersecting the $H$-cylinders in~\cite[Proposition~2.2]{Manzano} with $P$, one can easily verify that the surface $C_H$ corresponds to the choice
\begin{equation}\label{eqn:invariant-functions}
\begin{aligned}
r(u)&=\arctanh\left(\frac{\cos(u)}{2H}\right),\\
h(u)&=\frac{2H}{\sqrt{4H^2-1}}\arcsin\left(\frac{\sin(u)}{\sqrt{4H^2-\cos^2(u)}}\right).
\end{aligned}
\end{equation}
This parametrization is $2\pi$-periodic, and $u$ has been chosen such that the curve $\alpha_H$ has unit tangent vector $(-\sin u, \cos u)$. This follows from computing
\begin{equation}\label{eqn:invariant-angle-parameter}
  \alpha_H'(u) = \left( \frac{-2H\sin(u)}{4H^2-\cos^2(u)}, \frac{2H\cos(u)}{4H^2-\cos^2(u)} \right).
\end{equation}
Furthermore, $\alpha_H(u)$ is also convex in the $(r,h)$-plane since its Euclidean curvature with respect to the inward-pointing normal is
\begin{equation}\label{eqn:invariant-convexity}
\frac{r'(u)h''(u)-h'(u)r''(u)}{(r'(u)^2+h'(u)^2)^{3/2}}=\frac{4H^2-\cos^2(u)}{2H}>0,
\end{equation}
The surface $C_H$ is symmetric with respect to the totally geodesic surfaces $z=0$ and $x=0$, whose intersection is $\Gamma$, so we will say that $C_H$ is centered at $\Gamma$, see Figure~\ref{fig:foliation}. By means of the isometries of $\mathbb{H}^2\times\mathbb{R}$, we can find a unique family of horizontal $H$-cylinders centered at any horizontal geodesic $\Gamma\subset \mathbb{H}^2\times\mathbb{R}$.

The next two lemmas are directed to obtain two geometric conditions that will enable the application of the halfspace theorem~\cite[Theorem~7]{Mazet2013}, and the rest of its assumptions will be discussed directly in the proof of Theorem~\ref{thm:halfspace}.

\begin{lemma}\label{lm:foliation-CMC-Cylinders}
The family of horizontal $H$-cylinders $\{C_H:\frac{1}{2}<H<\infty\}$ centered at some horizontal geodesic $\Gamma\subset \mathbb{H}^2\times\mathbb{R}$ foliates $(\mathbb{H}^2\times\mathbb{R})-\Gamma$.
\end{lemma}

\begin{proof}
  We will assume that $\Gamma$ is the $y$-axis without losing generality, which reduces the problem to proving that the curves $\alpha_H(u)=(r(u),h(u))$ defined by~\eqref{eqn:invariant-functions} foliate $\R^2-\{(0,0)\}$ when $H$ ranges from $\frac{1}{2}$ to $+\infty$. Observe that $\alpha_H(u)$ is convex and its width and height in $\R^2$ diverge as $H\to\frac{1}{2}$, whilst it converges uniformly to $(0,0)$ as $H\to+\infty$. Therefore, the curves $\alpha_H$ for a large enough value of $H$ and for $H$ close to $\frac{1}{2}$ do not intersect. Let us reason by contradiction, assuming there is no such a foliation. Hence there exist $\frac{1}{2}<H_1<H_2<+\infty$ such that $\alpha_{H_1}$ and $\alpha_{H_2}$ are tangent at some point. Tangency implies that there is $u_0\in[0,2\pi]$ such that $\alpha_{H_1}(u_0)=\alpha_{H_2}(u_0)$ because of~\eqref{eqn:invariant-functions} and~\eqref{eqn:invariant-angle-parameter}. In particular, the value of $r(u_0)$ coincides for $H=H_1$ and $H=H_2$, i.e.,
\[\arctanh\left(\frac{\cos(u_0)}{2H_1}\right)=\arctanh\left(\frac{\cos(u_0)}{2H_2}\right).\]
Since $H_1<H_2$, we infer that $\cos(u_0)=0$, and hence we can assume by symmetry that $u_0=\frac{\pi}{2}$. Finally, taking into account that $h(u_0)=h(\frac{\pi}{2})=\frac{2H}{\sqrt{4H^2-1}}\arcsin(\frac{1}{2H})$ is a strictly increasing function of $H$, we conclude that it cannot give the same value for $H=H_1$ and $H=H_2$, and this gives the desired contradiction.
\end{proof}

\begin{lemma}\label{lemma:lag}
Horizontal $H$-cylinders have intrinsic linear area growth.
\end{lemma}

\begin{proof}
Assume that $C_H$ is centered at the $y$-axis, and consider $p_0=(0,1,h(\frac{\pi}{2}))\in C_H$, which also belongs to the vertical plane $P$ of equation $x^2+y^2=1$. Given $\rho>0$, the intrinsic metric ball $B_\rho(p_0)$ of $C_H$ centered at $p_0$ with radius $\rho$ is contained in the region of $C_H$ between two vertical planes at constant distance $\rho$ from $P$. This vertical slab is in turn contained in the slab between the two totally geodesic vertical planes $\Phi_{-\rho}(P)$ and $\Phi_\rho(P)$, which yields the estimate $\Area(B_\rho(p_0))\leq\Area(\phi([-\rho,\rho]\times[0,2\pi]))$, i.e.,
\[\Area(B_\rho(p_0))\leq\int_{-\rho}^\rho\left(\int_0^{2\pi}\frac{4H^2\df u}{(4H^2-\cos^2(u))^{3/2}}\right)\df t=2D\rho,\]
being $D>0$ the value of the integral in braces, which does not depend on $\rho$.
\end{proof}

\begin{proof}[Proof of Theorem~\ref{thm:halfspace}]
   If $4H^2+\kappa>0$, then existence follows from Theorem~\ref{thm:nodoids}. If $\kappa=H=0$, then nonexistence follows from Hoffman and Meeks' halfspace theorem for minimal surfaces in $\mathbb{R}^3$, see~\cite{HM1990}. Otherwise, we can assume, after rescaling the metric, that $\kappa=-1$, and argue by contradiction supposing the existence of a properly immersed $H_0$-surface $S\looparrowright\mathbb{H}^2\times\mathbb{R}$ with $0\leq H_0\leq\frac{1}{2}$, cylindrically bounded with respect to the $y$-axis in the halfspace model. Up to a vertical translation, we can also assume that $S\subset\mathbb{H}^2\times\mathbb{R}_+$, and consider the family of horizontal $H$-cylinders $C_H$ given by~\eqref{eqn:invariant-functions} centered at the horizontal geodesic $\Gamma=\{(0,y,-1):y>0\}$. Since they foliate $(\mathbb{H}^2\times\mathbb{R})-\Gamma$ by Lemma~\ref{lm:foliation-CMC-Cylinders} and $S$ is cylindrically bounded, there exists a cylinder $C_{H_{\text{out}}}$ for some $H_{\text{out}} > \frac{1}{2}$ such that $S\subset\Omega$, being $\Omega$ the intersection of $\mathbb{H}^2\times\mathbb{R}_+$ and the mean convex side of $C_{H_{\text{out}}}$. Thus $\Omega$ is foliated by the surfaces $C_H\cap\Omega$ with $H_{\text{out}}\leq H\leq H_{\text{in}}$ for some $H_{\text{in}}>\frac{1}{2}$ (see Figure~\ref{fig:foliation}).

  \begin{figure}[htbp]
    \centering
    \includegraphics{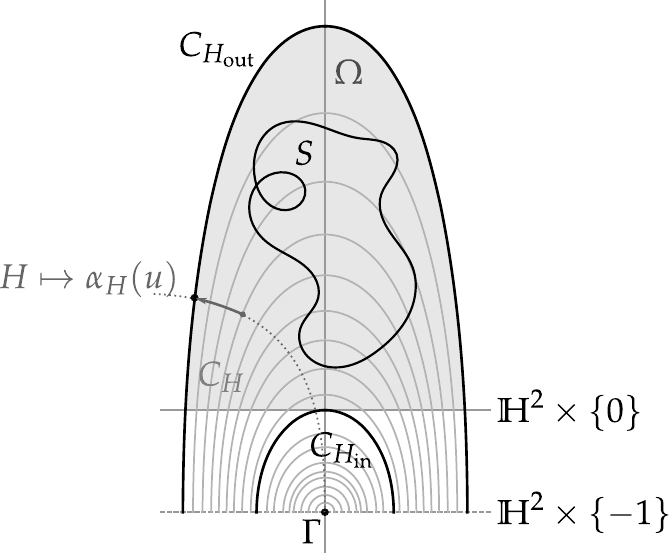}
    \caption{The foliation by horizontal cylinders $C_H$ centered at the geodesic $\Gamma = \{(0,y,-1)\colon y  >0\}$. The $H_0$-surface $S$ lies in the shaded region $\Omega$ contained in the mean convex side of $C_{H_\text{out}}$. The surface $C_{H_\text{in}}$ separates $\Omega$ and $\Gamma$. The curve $H\mapsto \alpha_H(u)$ in dotted lines represents the projection onto $C_{H_\text{out}}$.}
    \label{fig:foliation}
  \end{figure}

   Lemma~\ref{lemma:lag} yields the parabolicity of the leaves of the foliation. The surface $S$ lies in the mean convex side of the $C_{H_\text{out}}$ and has mean curvature $H_0 \leq \frac{1}{2}$, which is strictly less than the mean curvature of the leaves, with the right orientation in order to apply~\cite[Theorem~7]{Mazet2013}. This halfspace theorem implies that $S$ must be one of the cylinders in the foliation, and this is the contradiction we seek. It remains to verify that the foliation is regular in the sense of~\cite[Definition~5]{Mazet2013}. In the parametrization given by~\eqref{eqn:invariant-functions}, the shape operator $A$ of $C_H$ has norm
  \[
    \|A\|^2=4H^2-2\det(A)=\frac{3-16H^2+64H^4+4(1-4H^2)\cos(2u)+\cos(4u)}{16H^2},
  \]
  which is uniformly bounded for $H_{\text{out}}\leq H\leq H_{\text{in}}$. Moreover, the region $\Omega$ is geometrically bounded since $\mathbb{H}^2\times\mathbb{R}$ has bounded sectional curvature, so it suffices to show that there is a uniformly quasi-isometric projection sending all leaves $C_H\cap\Omega$ with $H_{\text{out}}\leq H\leq H_{\text{in}}$ onto $C_{H_{\text{out}}}\cap\Omega$. Note that all horizontal $H$-cylinders are $2\pi$-periodic in the parameter $u$ and they are invariant in the parameter $t$, so we can project onto $C_{H_{\text{out}}}\cap\Omega$ by just preserving the parameters $(t,u)$ given by~\eqref{eqn:invariant-parametrization}, i.e.\ by sending $\alpha_{H}(u)\mapsto\alpha_{H_{\text{out}}}(u)$ (see the dotted line in Figure~\ref{fig:foliation}). Periodicity in the parameter $u$ plus compactness of the interval $[H_{\text{out}},H_{\text{in}}]$ ensure this defines a uniformly quasi-isometric map.
\end{proof}

\section{The construction of horizontal nodoids}
\label{sec:conjugate-Plateau-technique}

This section is devoted to obtain the surfaces $\Sigma_{\lambda,H}^*$ of Theorem~\ref{thm:nodoids} as an extension of the construction of \emph{horizontal unduloids}. However, the arguments we will employ are significantly more involved than those in~\cite{MT}, since the fundamental piece is no longer a vertical graph for $\lambda>\frac\pi 2$. In the sequel we will omit the dependence on $H$, which will be fixed throughout the section.

\subsection{Conjugate immersions}\label{subsec:conjugate-Plateau-technique}
Given $\kappa,\tau\in\R$ such that $\kappa > 0$ and $\tau \neq 0$, the Berger sphere $\s^3_b(\kappa, \tau)$ is the usual $3$-sphere $\s^3 = \{(z, w) \in \C^2:\, \abs{z}^2 + \abs{w}^2 = 1\}$ equipped with the Riemannian metric
\[
  g(X, Y) = \tfrac{4}{\kappa}\bigl[\prodesc{X}{Y} + \bigl(\tfrac{4\tau^2}{\kappa} - 1\bigl)\prodesc{X}{V}\prodesc{Y}{V} \bigr],
\] 
where $\prodesc{\,}{\,}$ stands for the usual round metric in $\mathbb{S}^3$, and $V$ is the vector field defined by $V_{(z, w)} = (iz, iw)$. If $\kappa = 4\tau^2$, then $\s^3_b(4\tau^2, \tau)$ is a round sphere of constant sectional curvature $\tau^2$; otherwise, it is a homogeneous Riemannian manifold with isometry group of dimension $4$ (see~\cite[Section~2]{Torralbo} for more details). The Hopf fibration $\Pi: \s^3_b(\kappa, \tau) \rightarrow \s^2(\kappa) \subset \mathbb{R}^3$ given by $\Pi(z, w) = \frac{2}{\sqrt{\kappa}}\bigl(z\bar{w}, \frac{1}{2}(\abs{z}^2 - \abs{w}^2) \bigr)$ is a Riemannian submersion. The fibers of $\Pi$ are geodesics tangent to the unit Killing field $\widetilde\xi = \frac{\kappa}{4\tau}V$, and both the horizontal and vertical geodesics (with respect to $\Pi$) are great circles. We remark that the length of all vertical geodesics is $\frac{8\tau\pi}{\kappa}$, whereas the length of all horizontal geodesics is $\frac{4\pi}{\sqrt{\kappa}}$.

However, Berger spheres are not suitable for solving our Plateau problem (see Remark~\ref{rmk:polygon-case-Michigan}). We will use another space which is locally isometric to a Berger sphere but topologically different, namely the Riemannian three-manifold
\begin{equation}\label{eqn:metric-Ekt}
  M(\kappa,\tau) = \left( D_\kappa \times \mathbb{R}, \frac{\df x^2+\df y^2}{(1+\frac{\kappa}{4}(x^2+y^2))^2}+\left(\df z+\frac{\tau(y\df x-x\df y)}{1+\frac{\kappa}{4}(x^2+y^2)}\right)^2 \right),\quad 
\end{equation}
where $D_\kappa = \{(x, y) \in \mathbb{R}^2\colon 1 + \frac{\kappa}{4}(x^2 + y^2) > 0\}$, see~\cite{Daniel07}. There is a Riemannian covering map $\Theta:M(\kappa,\tau)\to\mathbb{S}^3_b(\kappa,\tau)-\{(e^{i \theta},0)\colon  \theta \in \mathbb{R}\}$, explicitly given by
\begin{equation}\label{eq:local-isometry-Daniel-Berger}
\begin{aligned}
\Theta(x, y, z) &= \frac{1}{\sqrt{1 + \tfrac{\kappa}{4}(x^2 + y^2))}} \left(\tfrac{\sqrt{\kappa}}{2}(x + iy) \exp(i \tfrac{\kappa}{4\tau}z), \exp(i \tfrac{\kappa}{4\tau} z)\right).\end{aligned}
\end{equation}
Hence $M(\kappa, \tau)$ is the universal cover of $\mathbb{S}^3_b(\kappa,\tau)$ minus a vertical fiber, and the lifted Hopf fibration $\Pi: M(\kappa,\tau) \to \mathbb{R}^{2}$ (also denoted by $\Pi$) acquires the simple form $\Pi(x, y, z) = (x, y)$. Although $M(\kappa, \tau)$ provides a unified model for all $\mathbb{E}(\kappa, \tau)$-spaces, it fails to be global or complete if $\kappa>0$.

In the discussion of the properties of conjugate surfaces we will make use of three types of minimal surfaces in Berger spheres as barriers: 
  \begin{itemize}
    \item The \emph{horizontal umbrella} centered at $p\in\mathbb{S}_b^3(\kappa,\tau)$ is the union of all horizontal geodesics going past $p$. Horizontal umbrellas are minimal spheres, but not every great sphere (that is, the intersection of a hyperplane of $\mathbb{R}^4$ with the $3$-sphere) is minimal with respect to the Berger metric. 

      The horizontal plane $z = \tfrac{4\tau}{\kappa} c$ in $M(\kappa, \tau)$ corresponds to a subset of the \emph{horizontal umbrella} centered at $(0, e^{ic})$ via~$\Theta$. 

    \item A \emph{Clifford torus} is the preimage of a geodesic of $\mathbb{S}^2(\kappa)$ by the Hopf fibration. Clifford tori have identically zero Gauss curvature, and they are the only minimal surfaces of $\mathbb{S}_b^3(\kappa,\tau)$ which are everywhere vertical.

   The vertical cylinders $(x -a)^2 + (y-b^2) = \frac{4}{\kappa} + (a^2 + b^2)$ in $M(\kappa, \tau)$, as well as the vertical planes $ax + by = 0$ containing the $z$-axis, are the minimal surfaces that correspond to Clifford tori via $\Theta$. 

  \item A \emph{spherical helicoid} is the minimal surface $\{(z, w) \in \mathbb{S}^3_b(\kappa, \tau)\colon \pIm(z \overline{w}^c) = 0\}$, $c \in [-1,1]$, obtained by moving a horizontal geodesic by a screw motion group of isometries along an intersecting vertical geodesic (see \cite[Section~4]{MT}). If $c = 0$, then the spherical helicoid is the minimal sphere $\pIm(z) = 0$, and if $c = 1$ then it is the Clifford torus $\pIm(z \overline{w}) = 0$.

  The Euclidean helicoid $\mathcal{H}_c = \{(x, y, z) \in \mathbb{R}^3\colon y = x \tan\left( \frac{\kappa}{4\tau}(c-1) z \right)\}$ corresponds to the spherical helicoid $\pIm(z \overline{w}^c) = 0$ via $\Theta$. 
\end{itemize}

Given a simply connected Riemannian surface $\Sigma$, as a particular case of Daniel sister correspondence, there is an isometric duality between minimal immersions $\widetilde{\phi}:\Sigma\to\mathbb{S}^3_b(4H^2+\kappa, H)$ and $H$-immersions $\phi:\Sigma\to\mathbb{M}^2(\kappa)\times\mathbb{R}$, as long as $H,\kappa\in\R$ satisfy $4H^2+\kappa>0$ and $H>0$. These immersions will be called \emph{conjugate} in the sequel and determine each other up to ambient isometries, see~\cite{Daniel07}. It should be noticed that the orientation of conjugate immersions must be compatible in the sense that there is a $\frac\pi2$-rotation $J$ in $T\Sigma$ such that $\{\df\widetilde\phi_p(v),\df\widetilde\phi_p(Jv),\widetilde N_p\}$ and $\{\df\phi_p(v),\df\phi_p(Jv), N_p\}$ are positively oriented in $\mathbb{S}^3_b(4H^2+\kappa, H)$ and $\mathbb{M}^2(\kappa)\times\mathbb{R}$, respectively, for all nonzero tangent vectors $v\in T_p\Sigma$. Here $\widetilde{N}$ is the unit normal to $\widetilde\phi$ with respect to which the mean curvature is computed and $N$ is the unit normal to $\phi$ defining the same \emph{angle function} $\nu\in\mathcal C^\infty(\Sigma)$, i.e., $\nu=\langle N,\xi\rangle = \langle\widetilde{N},\widetilde\xi\rangle$, where $\xi=\partial_t$ is the unit Killing vector field in $\mathbb{M}^2(\kappa)\times\R$ in the positive direction of the factor $\R$. In this particular case of Daniel correspondence, the $\frac\pi2$-rotation actually reflects some extrinsic geometric behaviour (see~\cite{Daniel07,MT,Plehnert2}):
\begin{itemize}
  \item The tangential projections $T = \xi - \nu N$ and $\widetilde{T} = \widetilde\xi - \nu \widetilde{N}$ of the unit Killing vector fields are intrinsically rotated by $\frac\pi2$, i.e., $\df \phi^{-1}(T) = J\df \widetilde\phi^{-1}(\widetilde{T})$, as well as the shape operators $S$ and $\widetilde{S}$ of the immersions  are related by $S = J \widetilde{S}$.

  \item Any horizontal or vertical geodesic curvature line in the initial surface becomes a plane line of symmetry in the conjugate one. Therefore, given a curve $\alpha$ in $\Sigma$, if $\widetilde{\phi}\circ\alpha$ is a horizontal (resp.\ vertical) geodesic, then $\phi \circ \alpha$ is contained in a totally geodesic vertical (resp.\ horizontal) surface, which the immersion meets orthogonally~\cite[Lemma~1]{MT}.
\end{itemize}
For the sake of simplicity, in the sequel we will use the notation $\widetilde\Sigma$ and $\Sigma$ for conjugate (immersed) surfaces. The surface $\widetilde\Sigma$ will be the solution of a Plateau problem over a geodesic polygon in $\mathbb{S}^3_b(4H^2+\kappa, H)$ consisting of vertical and horizontal geodesic segments making right angles at the vertexes. This guarantees that $\widetilde\Sigma$ can be smoothly extended across its boundary by successive axial symmetries about such geodesic segments to produce a complete smooth minimal immersion $\widetilde\Sigma^*$. Similarly, the conjugate immersed $H$-surface $\Sigma$ can be extended to a complete $H$-surface $\Sigma^*$ by means of mirror symmetries about totally geodesic horizontal and vertical planes in $\mathbb{M}^2(\kappa)\times\mathbb{R}$ containing the boundary components, see~\cite{MT}.

\subsection{Solving the Plateau problem}\label{subsec:plateau}
Assume that $H,\kappa\in\mathbb{R}$ are such that $H>0$ and $4H^2+\kappa>0$. For each $\lambda \ge 0$, consider the closed polygon $\widetilde\Gamma_\lambda\subset M(4H^2 + \kappa, H)$, consisting of three horizontal geodesics $\widetilde{h}_0$, $\widetilde{h}_1$ and $\widetilde{h}_2$, and one vertical geodesic $\widetilde v$, parametrized by the following expressions:
\[
\begin{aligned}
  \widetilde{h}_0(s)&= \left( \tfrac{2}{\sqrt{4H^2+\kappa}} \tfrac{\cos(2s)}{1 + \sin(2s)}, 0, 0\right)  ,& s &\in \bigl[0, \tfrac\pi2\bigr], \\
  \widetilde{h}_1(s)&=\left( \tfrac{2}{\sqrt{4H^2+\kappa}} \sin(2s), \tfrac{2}{\sqrt{4H^2+\kappa}} \cos(2s), \tfrac{4H}{4H^2+\kappa}\left(s - \tfrac{\pi}{4} \right)  \right) ,& s&\in\bigl[\tfrac\pi4, \tfrac\lambda2\bigr], \\
  \widetilde{h}_2(s)&=\left( \tfrac{2}{\sqrt{4H^2+\kappa}} \sin(2s), \tfrac{2}{\sqrt{4H^2+\kappa}} \cos(2s), \tfrac{4H}{4H^2+\kappa} (s + \tfrac{\pi}{4}) \right) ,& s&\in\bigl[-\tfrac\pi4, \tfrac\lambda2\bigr], \\
  \widetilde{v}(s)&= \left( \tfrac{2}{\sqrt{4H^2+\kappa}} \sin (\lambda), \tfrac{2}{\sqrt{4H^2+\kappa}} \cos (\lambda), \tfrac{4H}{4H^2+\kappa} ( s + \tfrac{\lambda}{2} - \tfrac{\pi}{4} )  \right) ,& s &\in \bigl[0, \tfrac\pi2\bigr].
\end{aligned}
\]
By abuse of the notation, $\widetilde{h}_0$, $\widetilde{h}_1$, $\widetilde{h}_2$, and $\widetilde v$ will be often treated as sets rather than parametrizations in the sequel. Moreover, $\widetilde\Gamma_\lambda$ is a geodesic quadrilateral whose vertexes will be labeled as $\widetilde 1$, $\widetilde 2$, $\widetilde 3$ and $\widetilde 4$, as shown in Figure~\ref{fig:polygon-Berger}. 

\begin{figure}[htbp]
\centering
\includegraphics{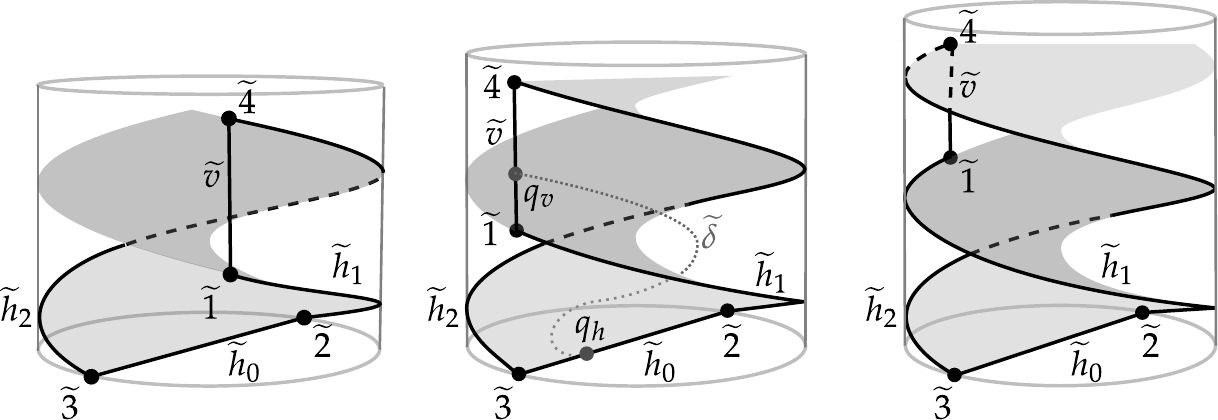}
\caption{A faithful representation of the polygon $\widetilde{\Gamma}_\lambda$ for different values of $\lambda$. The barriers $T$ (vertical cylinder) and $S$ (helicoid) demarcate the mean convex solid $\Omega$. The dotted line (see central figure) represents the curve $\widetilde{\delta}$ of zeroes of the angle function defined in Proposition~\ref{prop:angle}.}
\label{fig:polygon-Berger}
\end{figure}

\begin{remark}\label{rmk:polygon-case-Michigan}
  For each $\lambda \in [0, \frac{\pi}{2}]$, the polygon $\Theta(\widetilde{\Gamma}_\lambda)\subset\mathbb{S}^3_b(4H^2+\kappa,H)\subset\C^2$ is, up to the isometry $(z, w) \mapsto \frac{1}{\sqrt{2}}\bigl(e^{-i\frac{\pi}{4}}(z + i w), e^{i \frac{\pi}{4}}(z - i w)\bigr)$, the same as in the construction of the horizontal unduloids~\cite[Secion~5.1]{MT}. However, the barriers we used to solve the Plateau problem in~\cite{MT} are no longer valid if $\lambda>\frac12$. Furthermore, the polygon $\Theta(\widetilde{\Gamma}_\lambda)$ has self-intersections if $\lambda\geq\frac{7\pi}{2}$, so the resulting Plateau problem is ill-posed in $\mathbb{S}^3_b(4H^2+\kappa,H)$, and this is the reason why we use the locally isometric model $M(4H^2+\kappa,H)$ throughout this section.
\end{remark}

Let $T$ be the vertical minimal cylinder that corresponds to the Clifford torus $|z|^2=|w|^2$ in $\mathbb S^3\subset\mathbb C^2$, and let $S$ be the minimal helicoid $\mathcal H_{-1}$. In the model $M(4H^2+\kappa,H)$, these surfaces are given by
\begin{equation*}
\begin{aligned}
 T&=\left\{(x,y,z)\in\R^3\colon x^2 + y ^2 = \tfrac{4}{4H^2+\kappa}\right\}\\
 S&=\left\{(x, y, z) \in \mathbb{R}^3\colon y = -x\tan\left(\tfrac{4H^2+\kappa}{2H} z\right)\right\}.
\end{aligned}
\end{equation*}
The interior domain of the cylinder will be denoted by $G$, and it is divided by $S$ in two connected components. The closure of the component that contains $\widetilde{v}$ is
\[
  \Omega = \left\{(x, y, z)\in\mathbb R^3\colon x^2 + y^2 \leq \tfrac{4}{4H^2+\kappa},\, y \geq -x \tan\left(\tfrac{4H^2+\kappa}{2H}z\right)\right\},
\] 
and satisfies that $\widetilde{\Gamma}_\lambda \subset \partial{\Omega}$ for any $\lambda \ge 0$. Besides, $\Omega$ is a mean-convex solid in the sense of Meeks and Yau~\cite{MY82} so the Plateau problem with boundary $\widetilde\Gamma_\lambda$ can be solved. This produces an embedded closed minimal disk $\widetilde\Sigma_{\lambda}\subset\Omega$ with boundary $\partial\widetilde\Sigma_{\lambda}=\widetilde\Gamma_\lambda$, which will play the role of the initial minimal surface in the conjugate construction. We will show in Proposition~\ref{prop:uniqueness} that the solution $\widetilde{\Sigma}_{\lambda}$ is unique and hence depends continuously on $\lambda$.

We highlight the following special cases, depicted in {Figure~\ref{fig:extremal-cases}:
\begin{itemize}
  \item If $\lambda = 0$, then $\Theta(\widetilde{\Sigma}_{0})$ is part of the spherical helicoid $\pIm(z^2 + w^2) = 0$.
  \item If $\lambda = \frac{\pi}{2}$, then $\Theta(\widetilde{\Sigma}_{\frac{\pi}{2}})$ is part of the minimal sphere $\pIm(z - w) = 0$.
  \item If $\lambda = \frac{3\pi}{2}$, then $\Theta(\widetilde{\Sigma}_{\frac{3\pi}{2}})$ is part of (the Berger-sphere version of) Lawson's Klein bottle $\eta_{1,1}\subset\s^3$, see~\cite[Theorem 2]{Torralbo}. 
\end{itemize}

\begin{figure}[htbp]
  \centering
  \includegraphics{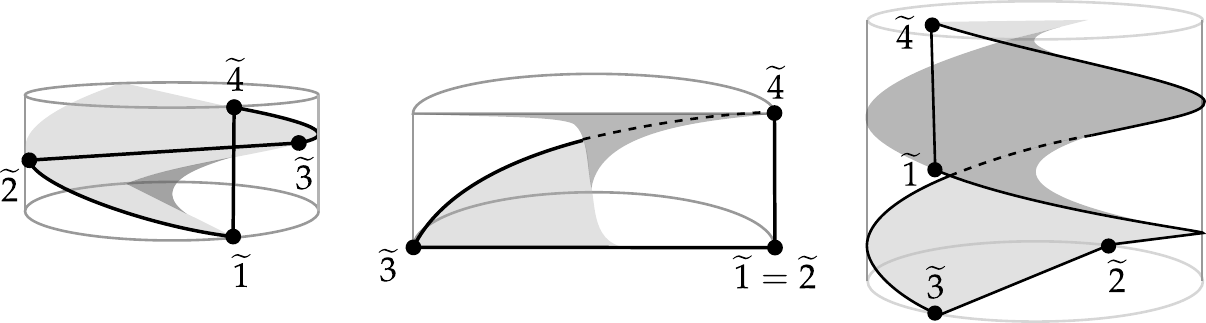}
  \caption{From left to right: polygon $\widetilde{\Gamma}_\lambda$ for $\lambda = 0$ ($\widetilde{\Sigma}_0$ is a spherical helicoid with axis $\widetilde{14}$), $\lambda = \frac{\pi}{2}$ ($\widetilde{\Sigma}_{\frac{\pi}{2}}$ is the horizontal umbrella centered at $\widetilde{3}$) and $\lambda = \frac{3\pi}{2}$ where $\widetilde{3}$, $\widetilde{1}$ and $\widetilde{4}$ are on the same vertical geodesic ($\widetilde{\Sigma}_{\frac{3\pi}{2}}$ is the Lawson's Klein bottle $\eta_{1,1}$). 
  }
  \label{fig:extremal-cases}
\end{figure}

\begin{remark}\label{rmk:round-sphere1}
If $\kappa=0$, the Berger sphere $\mathbb S_b^3(4H^2+\kappa,H)$ is the three-sphere $\mathbb{S}^3(H^2)$ of constant sectional curvature $H^2$, and the lengths of the geodesic segments $\widetilde{14}$ and $\widetilde{23}$ coincide (they are a quarter of the length of a great circle of $\mathbb{S}^3(H^2)$). The completion of $\widetilde\Sigma_{\lambda}$ is invariant under a continuous $1$-parameter family of screw motions, which are composition of translations and suitable rotations about $\widetilde{34}$. In particular, $\widetilde\Sigma_{\lambda}$ is an equivariant minimal surface if $\kappa=0$.  However, if $\kappa\neq 0$, this argument fails since there are no screw motions with axis $\widetilde{34}$, and the geodesic arcs $\widetilde{14}$ and $\widetilde{23}$ (see Figure~\ref{fig:extremal-cases}) have different lengths.
\end{remark}

\subsection{The analysis of the angle function}
Next proposition gives some insight into the behaviour of the angle function $\nu_\lambda: \widetilde\Sigma_{\lambda} \to [-1,1]$, whose sign is chosen such that $\nu_\lambda(\widetilde{2})=1$. It will be fundamental in the study of the conjugate surface.

\begin{proposition}\label{prop:angle} 
Let $\widetilde{\Sigma}_\lambda$ be the compact minimal disk spanning $\widetilde{\Gamma}_\lambda$ with $\lambda>\frac{\pi}{2}$, and consider the angle function $\nu_\lambda$ of $\widetilde\Sigma_\lambda$ such that $\nu_\lambda(\widetilde 2) = 1$.
\begin{enumerate}[label=(\alph*)]
  \item The only points in which $\nu_\lambda$ takes the values $1$ and $-1$ are $\widetilde 2$ and $\widetilde 3$, respectively.

  \item \label{prop:angle:item:zeroes-angle} The set of points in which $\nu_\lambda$ vanishes consists of $\widetilde v$ and a certain interior regular curve $\widetilde{\delta}\subset\widetilde\Sigma_\lambda$ with endpoints in $\widetilde v$ and $\widetilde h_0$ (see Figure~\ref{fig:polygon-Berger}).

  \item Given $p\in\widetilde h_0((0,\frac{\pi}{2}))\cup\widetilde h_1((\frac{\pi}{4},+\infty))\cup\widetilde h_2((\frac{-\pi}{4},+\infty))$, the function $\lambda\mapsto\nu_\lambda(p)$ is continuous in the interval where it is defined.
  \begin{itemize}
    \item It is strictly increasing (possibly changing sign) if $p \in\widetilde h_0((0,\frac{\pi}{2}))$.
    \item It is positive and strictly increasing if $p \in\widetilde h_1((\frac{\pi}{4},+\infty))$.
    \item It is negative and strictly decreasing if $p \in\widetilde h_2((\frac{-\pi}{4},+\infty))$. 
  \end{itemize}
\end{enumerate}
\end{proposition}

The proof of items (a) and (b) essentially relies on comparing $\widetilde\Sigma_\lambda$ with two types of surfaces, $U_p$ and $T_p$, tangent to $\widetilde\Sigma_\lambda$ at some $p\in\widetilde\Sigma_\lambda$. On the one hand, if $\nu_\lambda(p) ^2=1$, consider the umbrella $U'_p$ tangent to $\widetilde{\Sigma}_{\lambda}$ at $p$, and define $U_p$ as the closure of the connected component of $U'_p\cap G$ that contains $p$. The interior of $U_p$ is a vertical graph in $G$, and if $\widetilde v$ lies in $\partial U_p\subset T$ then $\Pi(p)$ and $\Pi(\widetilde v)$ are opposite points of the great circle $\Pi(T)\subset\R^2$ and $p \in \partial \widetilde{\Gamma}_{\lambda}$. On the other hand, if $\nu_\lambda(p)=0$, consider the Clifford torus $T'_p$ tangent to $\widetilde\Sigma_\lambda$ at an interior point $p$, and define $T_p$ as the closure of the connected component of $(T'_p\cap G)-S$ containing $p$. Note that $T_p$ is a vertical quadrilateral with boundary in $S\cup T$: three of its sides lie in $S$ if $T_p'$ contains the $z$-axis (see Figure~\ref{fig:intersection-helicoid-torus} center), otherwise only two of the sides lie in $S$ (see Figure~\ref{fig:intersection-helicoid-torus} left and right). 

\begin{figure}[htbp]
  \centering
  \includegraphics{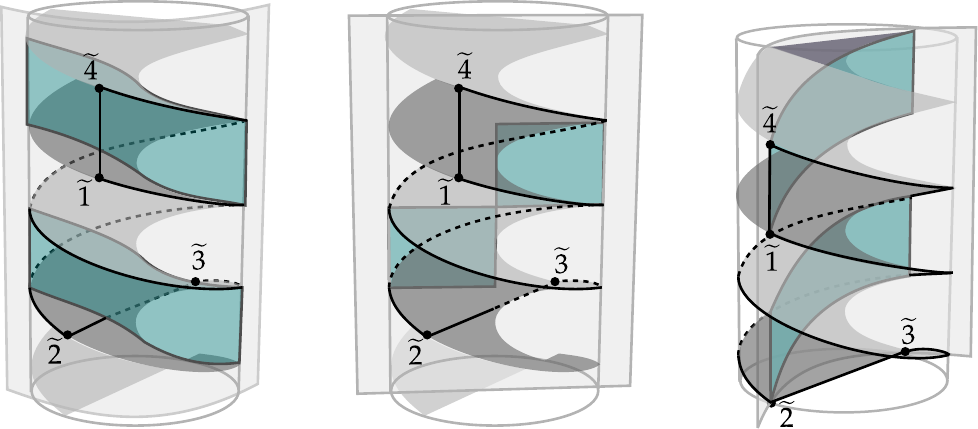}
  \caption{Each figure indicates two possible connected components $T_p$ (in turquoise) for a Clifford torus $T_p'$ inside the mean convex solid $\Omega$. From left to right: a general case, a case in which $T_p'$ contains the axis of $S$, and a case in which $T_p'$ contains $\widetilde{v} = \widetilde{14}$.}
  \label{fig:intersection-helicoid-torus}
\end{figure}

\begin{lemma}\label{lemma:foliation}
  Assume that $p\in\widetilde\Sigma_\lambda$ is an interior point with $\nu_\lambda(p)=0$. Then the intersection of $\widetilde\Sigma_\lambda$ and the vertical quadrilateral $T_p$ does not contain any closed curve.  
\end{lemma}

\begin{proof}
  Let $T_p'$ be the Clifford torus containing $T_p$. Notice that $T_p'\neq T$ because otherwise $\widetilde\Sigma_\lambda$ and $T$ are tangent at the interior point $p$, in contradiction with the maximum principle. Write $T\cap T_p'=\gamma_1\cup\gamma_2$, where $\gamma_1$ and $\gamma_2$ are vertical geodesics whose projections by $\Pi$ are opposite points in $\Pi(T)\subset\R^2$. Assume by contradiction that there is a closed curve $\beta\subset T_p\cap\widetilde\Sigma_\lambda$. Taking into account that $\widetilde\Sigma_\lambda$ is a disk, that implies the existence of another disk $D \subset\widetilde\Sigma_\lambda$ with boundary $\beta$.

  Observe that $\beta$ lies in the intersection of minimal surfaces, and can be assumed a piecewise smooth Jordan curve with with interior positive angles at its vertexes; in particular, $\widetilde\Sigma_\lambda$ must be vertical at such vertexes. This implies that $\beta$ cannot reach $\widetilde 2$ or $\widetilde 3$, where $\widetilde\Sigma_\lambda$ is horizontal. It cannot reach any interior point of $\widetilde h_1$ or $\widetilde h_2$ either, since this goes against the boundary maximum principle when one compares $\widetilde\Sigma_\lambda$ and $T$ at such a point. On the other hand, if $\beta\cap\widetilde v\neq\emptyset$, it follows that $\widetilde v\subset T_p\cap\gamma_1$ (swapping $\gamma_1$ and $\gamma_2$ if necessary) and thus $\beta\cap\gamma_1\neq\emptyset$. Since $\beta$ does not intersect $\widetilde h_1$ or $\widetilde h_2$, this implies that $\beta\cap\gamma_2=\emptyset$, and hence we can move $D$ inside $G$ away from $\gamma_1\cup\gamma_2$ by an isometry preserving $T_p'$, e.g., the lift to $M(\kappa,\tau)$ of a sufficiently small translation along $\Pi(T_p')$. 

  The previous paragraph implies that we can assume that $\beta$ does not touch $\widetilde h_1\cup\widetilde h_2\cup\widetilde v$, and hence $\beta\cap T=\emptyset$. The family of vertical cylinders containing $\gamma_1\cup\gamma_2$ is an open-book minimal foliation of $G$, whence the maximum principle with respect to this foliation implies that $D$ (and hence $\widetilde\Sigma_\lambda$) is contained in one of the cylinders, which is the desired contradiction. It is important to notice that there could be points of the interior of $\widetilde h_0$ in $T_p$, but they do not belong to $T$ and therefore do not concern the above arguments.
\end{proof}

We also need to understand the local picture around a point in which the angle function and some of its derivatives vanish. Next lemma is stated for an arbitrary $\E(\kappa,\tau)$-space, and we believe it is true for all $n\in\N$.

\begin{lemma}\label{lemma:intersection}
Let $\Sigma\subset\mathbb{E}(\kappa,\tau)$ be a minimal surface whose angle function satisfies $\nu(p)=0$ at some $p\in\Sigma$, and let $T$ be the vertical minimal cylinder tangent to $\Sigma$ at $p$. If the derivatives of $\nu$ at $p$ vanish up to order $n\in\{1,2\}$, then $T_p\cap\Sigma$ consists of at least $n+2$ curves meeting transversally at $p$.
\end{lemma}

\begin{proof}
The intersection of two minimal surfaces is a family of regular curves meeting transversally at points where the surfaces are tangent, and the number of curves at such points is the order of contact plus one. As the argument is local, we will use the model given by~\eqref{eqn:metric-Ekt} and assume that $p=(0,0,0)$ and $T=\{(x,y,z)\in D_\kappa\times\R:x=0\}$ up to an ambient isometry. Therefore, $\Sigma$ can be expressed as $x=f(y,z)$ around $p$, and the tangency condition reads
\begin{equation}\label{lemma:intersection:eqn1}
f(0,0)=0,\quad f_y(0,0)=0,\quad f_z(0,0)=0.
\end{equation}
It is a harsh computation to work out the mean curvature $H(y,z)$ and the angle function $\nu(y,z)$ of $\Sigma$. Since we are only interested on their values at $p$, we can employ~\eqref{lemma:intersection:eqn1} to simplify the calculations and to obtain
\begin{equation}\label{lemma:intersection:eqn2}
\begin{aligned}
H(0,0)&=\tfrac{1}{2}\left(f_{yy}(0,0)+f_{zz}(0,0)\right) = 0,\\
\nu_y(0,0)&=-f_{yz}(0,0),\qquad\nu_z(0,0)=-f_{zz}(0,0).
\end{aligned}
\end{equation}
From~\eqref{lemma:intersection:eqn2}, we deduce that $\nabla\nu(p)=0$ if and only if $f_{yy}(0,0)=f_{yz}(0,0)=f_{zz}(0,0)=0$, i.e., if and only if $\Sigma$ and $T$ coincide up to the second order at $p$.

If the derivatives of $\nu$ vanish up to the second order, then the derivatives of $f$ also vanish up to the second order. Another long computation using this yields
\begin{equation}\label{lemma:intersection:eqn3}
\nu_{yy}(0,0)\!=\!-f_{yyz}(0,0),\quad
\nu_{yz}(0,0)\!=\!-f_{yzz}(0,0),\quad
\nu_{zz}(0,0)\!=\!-f_{zzz}(0,0).
\end{equation}
Taking derivatives in the expression of the mean curvature and evaluating at $(0,0)$, we also get that the following must vanish:
\begin{equation}\label{lemma:intersection:eqn4}
H_y(0,0)=\tfrac{1}{2}(f_{yyy}(0,0)+f_{yzz}(0,0)),\quad
H_z(0,0)=\tfrac{1}{2}(f_{yyz}(0,0)+f_{zzz}(0,0)).
\end{equation}
From~\eqref{lemma:intersection:eqn3} and~\eqref{lemma:intersection:eqn4}, it follows that $\nabla\nu(p)=0$ and $\nabla^2\nu(p)=0$ if and only if $\Sigma$ and $T$ coincide up to the third order at $p$.
\end{proof}

\begin{proof}[Proof of Proposition~\ref{prop:angle} ]
  Since two horizontal geodesics meet at $\widetilde{2}$ and $\widetilde{3}$, it is clear that $\nu_\lambda$ equals $\pm 1$ at these points.  The choice $\nu_\lambda(\widetilde{2})=1$ implies that $\widetilde N$, the unit normal to $\widetilde\Sigma_\lambda$, points towards the interior of one of the components of $G-(S\cup\widetilde\Sigma_\lambda)$. It follows that $\nu_\lambda$ is positive along $\widetilde h_1$ and negative along $\widetilde h_2$, and cannot be zero in the interior of $\widetilde h_1$ or $\widetilde h_2$ by the boundary maximum principle with respect to $T$, whose angle function is identically zero (see Claim 1 below). 

Conversely, assume that $p\in\widetilde\Sigma_\lambda$ is such that $\nu_\lambda(p)^2=1$ which is not a vertex $\widetilde\Gamma_\lambda$, and let us reach a contradiction. In particular  $p$ does not belong to $\widetilde v$ (along which $\nu_\lambda$ vanishes). The minimal surfaces $U_p$ and $\widetilde\Sigma_\lambda$ are tangent at $p$, so their intersection contains (at least) two curves meeting transversally at $p$. We will finish the proof of item (a) by distinguish two cases:
\begin{enumerate}
  \item If $\lambda>\pi$ and $p=\widetilde h_2(\frac{\lambda-\pi}{2})$, then $\partial U_p$ contains $\widetilde v$ and part of $\widetilde h_2$, whereas $\widetilde h_0\cup\widetilde h_1$ is under $U_p$. This means that $U_p$ can be used as a barrier in the solution of the Plateau problem and hence $U_p\cap\widetilde\Sigma_\lambda\subset \widetilde v\cup\widetilde h_2$. This contradicts the above assertion that there are at least two curves in the intersection around $p$.
   \item Otherwise, no matter whether $p$ is in the interior or in the horizontal boundary of $\widetilde\Sigma_\lambda$, $\partial U_p$ does not intersect $\widetilde v$, and there are interior curves in the intersection $U_p\cap\widetilde\Sigma_\lambda$ with endpoints in $U_p\cap\widetilde\Gamma_\lambda$. Some of the horizontal geodesics joining these endpoints with $p$ (which are contained in $U_p$ by definition of umbrella but not necessarily in $\widetilde\Sigma_\lambda$), together with part of $\widetilde{h}_0\cup\widetilde{h}_1\cup\widetilde{h}_2$, form a closed horizontal geodesic triangle or quadrilateral in $\overline G$. It projects injectively via the Hopf fibration $\Pi$ to the boundary a certain geodesic triangle or quadrilateral in an hemisphere of $\mathbb{S}^2(4H^2+\kappa)$. This contradicts the fact that the bundle curvature is not zero (see~\cite[Proposition~3.3]{Man14}).
 \end{enumerate}
As for item (b), since $\nu_\lambda$ lies in the kernel of the stability operator of $\widetilde\Sigma_\lambda$, the nodal set $Z=\{p\in\widetilde\Sigma_\lambda:\nu_\lambda(p)=0\}$ forms a set of regular curves with endpoints in $\widetilde\Gamma_\lambda$ and transverse intersections precisely at points with $\nabla\nu_\lambda=0$. Observe that $\widetilde{v}\subset Z$ since $\widetilde v$ is a vertical geodesic, but $Z$ must also contain other components because $\nu_\lambda$ changes sign along $\widetilde h_0$ in view of item (a). We will prove that points of $Z$ not lying in $\widetilde v$ belong to a single regular curve $\widetilde\delta$ joining a certain point of $q_v\in\widetilde v$ (such that $\nabla\nu_\lambda(q_v)=0$) and some point $q_h\in\widetilde h_0$ (at which the change of sign takes place). This is a consequence of the following five claims:

  \textbf{Claim 1.} \emph{There are no points in $\widetilde h_1$ or $\widetilde h_2$ with $\nu_\lambda=0$ rather than $\widetilde 1$ or $\widetilde 4$.}

  If there are points in the interior of $\widetilde h_1$ or $\widetilde h_2$ with $\nu_\lambda = 0$ then $\widetilde\Sigma_\lambda$ is tangent to the Clifford torus $T$, but this goes against the boundary maximum principle for minimal surfaces. 

  \textbf{Claim 2.} \emph{There is exactly one point $q_h\in\widetilde h_0$ such that $\nu_\lambda(q_h)=0$.}

  Reasoning by contradiction, assume there are $p,q\in\widetilde h_0$ such that $\nu_\lambda(p)=\nu_\lambda(q)=0$ so the tangent cylinders $T_p'$ and $T_q'$ contained the $z$-axis. This means that there is an interior curve $\gamma_p\subset T_p\cap\widetilde\Sigma_\lambda$ (resp.\ $\gamma_q\subset T_q\cap\widetilde\Sigma_\lambda$) with one endpoint equal to $p$ (resp.\ $q$) and the other endpoint in $T_p\cap\widetilde\Gamma_\lambda$ (resp.\ $T_q\cap\widetilde\Gamma_\lambda$). 
  Since $T_p$ (resp. $T_q$) contains the $z$-axis, there are two possibilities for $T_p$ (resp.\ $T_q$), one of them containing $\widetilde 2$ and the other one containing $\widetilde 3$. The latter is not possible since it is not contained in the mean convex body $\Omega$. It follows that $T_p=T_q$ have one vertical side over the $z$-axis and $\widetilde 2$ belongs to the other vertical side, whilst the other two sides are horizontal geodesics lying in $S$. This implies that $T_p\cap\widetilde\Gamma_\lambda=T_q\cap\widetilde\Gamma_\lambda$ consists of a half of $\widetilde h_0$ and one isolated point in $\widetilde h_2$, and hence one can find a closed curve in $\gamma_p\cup\gamma_q\cup\widetilde h_0$ in contradiction with Lemma~\ref{lemma:foliation}.

  \textbf{Claim 3.} \emph{There are no interior points of $\widetilde\Sigma_\lambda$ in which $\nu_\lambda=0$ and $\nabla\nu_\lambda=0$.}

  Reasoning by contradiction, assume there is such an interior point $p$, and consider the vertical quadrilateral $T_p$ tangent to $\widetilde\Sigma_\lambda$ at $p$, see Figure~\ref{fig:intersection-helicoid-torus}. By Lemma~\ref{lemma:intersection}, $T_p\cap\widetilde\Sigma_\lambda$ contains (at least) three curves meeting transversally at $p$ with (at least) six endpoints in $\partial T_p\cap\widetilde\Gamma_\lambda$. If two of the endpoints lie in $\widetilde v$, then the corresponding two curves, along with a segment of $\widetilde v$, form a closed curve in $T_p$ contradicting Lemma~\ref{lemma:foliation}. If two of the endpoints lie in $\widetilde h_0$, then either they coincide (and the corresponding two curves again contradict Lemma~\ref{lemma:foliation}) or they are different (and hence $T_p$ contains part of $\widetilde h_0$ so $\nu_\lambda=0$ at two different points, in contradiction with Claim 2). This means that $T_p$ intersects each of the curves $\widetilde h_0$, $\widetilde v$ at most once, and $\widetilde h_1$ and $\widetilde h_2$ at most twice. Nonetheless, it cannot intersect $\widetilde h_1$ and $\widetilde h_2$ twice and also $\widetilde h_0$, whence at least two of the six endpoints of the three curves meeting at $p$ coincide by the pidgeonhole principle. The corresponding curves from $p$ form a closed curve in $T_p$ that goes against Lemma~\ref{lemma:foliation}.

  \textbf{Claim 4.} \emph{There is exactly one point in $q_v\in\widetilde v$ such that $\nabla \nu_\lambda(q_v)=0$.}

  Reasoning by contradiction, assume that there exist $p,q\in\widetilde v$ such that $\nabla \nu_\lambda(p)=\nabla \nu_\lambda(q)=0$. Thus, consider the vertical quadrilaterals $T_p$ and $T_q$ tangent to $\widetilde\Sigma_\lambda$ at $p$ and $q$, respectively. By Lemma~\ref{lemma:intersection}, there are (at least) three curves in $T_p\cap\widetilde\Sigma_\lambda$ (resp. $T_q\cap\widetilde\Sigma_\lambda$) meeting transversally at $p$ (resp.\ $q$), being one of them $\widetilde v$ itself. Let us distinguish two cases:
  \begin{enumerate}
    \item If $\frac{\pi}{2}<\lambda<\frac{3\pi}{2}$, then both $(T_p\cap\widetilde\Gamma_\lambda)-\widetilde v$ and $(T_q\cap\widetilde\Gamma_\lambda)-\widetilde v$ consist of at most one point in $\widetilde h_2$ and one point in $\widetilde h_0$, we deduce that the two interior curves $\gamma_p^0,\gamma_p^2\subset T_p\cap\widetilde\Sigma_\lambda$ and $\gamma_q^0,\gamma_q^2\subset T_q\cap\widetilde\Sigma_\lambda$ can be chosen such that $\gamma_p^i$ and $\gamma_q^i$ have endpoints in $\widetilde h_i$ for $i\in\{0,2\}$ (see Figure~\ref{fig:claim3} left). Note that if both curves arrived in $\widetilde h_0$ or $\widetilde h_2$, then $T_p$ or $T_q$ would be tangent to the Clifford torus containing $\widetilde h_0$ or $\widetilde h_2$, which is obviously not possible. Hence $\gamma_p^i$ and $\gamma_q^i$ end at $\widetilde{h}_i$ for each $i\in\{1,2\}$. Notice that $\gamma_p^2$ and $\gamma_q^2$ have the same endpoint in $\widetilde h_2$ because the projections $\Pi(T_p),\Pi(T_q)\subset\mathbb{S}^2(4H^2+\kappa)$ are great circles, which intersect at two antipodal points (see Figure~\ref{fig:claim3} right). Assume without loss of generality that $p$ is closer to $\widetilde 4$ than $q$, so the curves $\gamma_p^0$ and $\gamma_q^2$ intersect at some interior point. This implies that $T_p=T_q$ and there exists a closed curve contained in $\widetilde v\cup\gamma_p^0\cup\gamma_q^2\subset T_p$, which contradicts Lemma~\ref{lemma:foliation}.
    \item If $\lambda\geq\frac{3\pi}{2}$, then $(T_p\cap\widetilde\Gamma_\lambda)-\widetilde v$ and $(T_q\cap\widetilde\Gamma_\lambda)-\widetilde v$ consist of at most one point in $\widetilde h_2$ and one point in $\widetilde h_1$ rather than $\widetilde h_0$. The reasoning in item (1) can be mimicked by substituting $\widetilde h_0$ with $\widetilde h_1$. Note that the new curves $\gamma_p^1$ and $\gamma_q^1$ ending in $\widetilde h_1$ do have the same endpoint in this case.
  \end{enumerate}
  \begin{figure}[htbp]
    \centering
    \includegraphics[width=\textwidth]{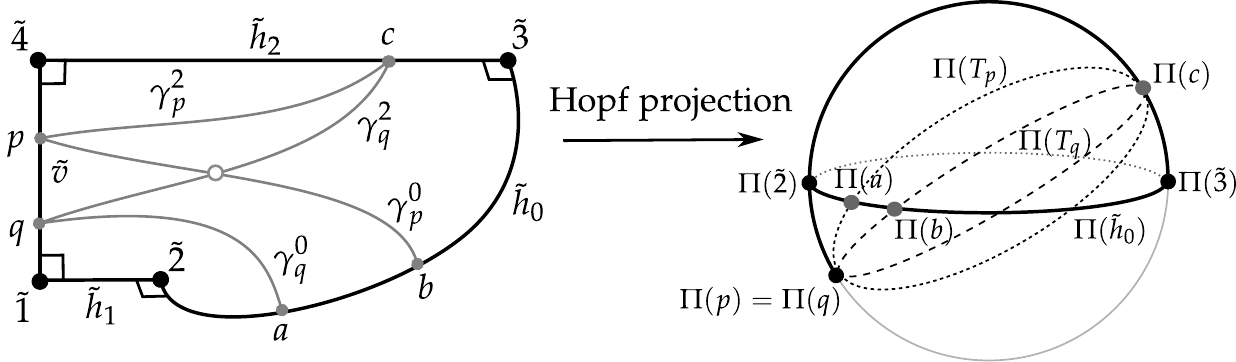}
    \caption{Case $\frac{\pi}{2}<\lambda<\frac{3\pi}{2}$ in Claim 4: schematic representation of the geodesic polygon $\widetilde{\Gamma}_\lambda$ (left) and its Hopf projection to $\mathbb{S}^2(4H^2 + \kappa)$ (right). The Clifford tori $T_p$ and $T_q$ project to great circles containing $\Pi(p) = \Pi(q)$ so they intersect the polygon $\widetilde{\Gamma}_\lambda$ in two points $a$ and $b$ along $\widetilde{h}_0$ and in the same point $c$ along $\widetilde{h}_2$}
    \label{fig:claim3}
  \end{figure}

  \textbf{Claim 5.} \emph{There is exactly one interior curve $\widetilde\delta\subset\widetilde\Sigma_\lambda$ where $\nu_\lambda$ vanishes.}\nopagebreak

  Due to the previous claims, it suffices to show that at the point $q_v\in\widetilde v$ (given by Claim 4) no more than two curves of $Z$ meet (being $\widetilde v$ one of them), and at the point $q_h\in\widetilde h_0$ (given by Claim 2), there is only one curve, i.e., $\nabla\nu_\lambda(q_h)\neq 0$ (see Figure~\ref{fig:polygon-Berger}). On the one hand, if there are (at least) two interior curves of $Z$ meeting at $q_v$, then the vertical quadrilateral $T_{q_v}$ intersects $\widetilde\Sigma_\lambda$ in at least four curves by Lemma~\ref{lemma:intersection} (case $n=2$), one of them being $\widetilde v$, so there are at least three interior curves in $\widetilde\Sigma_\lambda\cap T_{q_v}$ around $q_v$. Since $T_{q_v}$ intersects $\widetilde\Gamma_\lambda$ in $\widetilde v$, in one point of $\widetilde h_2$ and in at most one point of either $\widetilde h_0$ or $\widetilde h_1$, it follows that there is a minimal disk contained in $\widetilde\Sigma_\lambda$ with boundary in $T_{q_v}$ contradicting Lemma~\ref{lemma:foliation} as in previous claims. On the other hand, if $\nabla\nu_\lambda(q_h)=0$, then $T_{q_h}$ intersects $\widetilde\Sigma_\lambda$ in at least three curves by Lemma~\ref{lemma:intersection}, being $\widetilde h_0$ one of them. Since $T_{q_h}\cap\widetilde\Gamma_\lambda$ consists of $\widetilde h_0$ and one point in $\widetilde h_2$, this leads to the same contradiction as in the case of $q_v$.

As for item (c), let $\frac{\pi}{2}\leq\lambda_1<\lambda_2$, and observe that $\widetilde{h}_1((\frac{\pi}{4},\frac{\lambda_1}{2}))\subset \widetilde{h}_1((\frac{\pi}{4}, \frac{\lambda_2}{2}))$  and $\widetilde{h}_2((\frac{-\pi}{4}, \frac{\lambda_1}{2})) \subset \widetilde{h}_2((-\frac{\pi}{4}, \frac{\lambda_2}{2}))$, whereas $\widetilde h_0$ does not depend on $\lambda$. This means that, for each $p$ in the horizontal boundary, the function $\lambda\mapsto\nu_\lambda(p)$ is defined and continuous on an interval of the form $[\lambda_0,+\infty)$ for some $\lambda_0$ depending on $p$. As $\widetilde\Gamma_{\lambda_1}$ lies in the boundary of the mean convex open subset of $\Omega$ bounded by $S$, $T$ and $\widetilde\Sigma_{\lambda_2}$, the surface $\widetilde\Sigma_{\lambda_2}$ can be seen as a barrier for $\widetilde\Sigma_{\lambda_1}$, and hence $\widetilde\Sigma_{\lambda_1}$ and $\widetilde\Sigma_{\lambda_2}$ are ordered along their common boundary. Since the angle function does not take values $\pm 1$ in the interior of the horizontal boundary components, the monotonicity properties in item (c) follow from comparing the normal vector fields to $\widetilde\Sigma_{\lambda_1}$ and $\widetilde\Sigma_{\lambda_2}$ along their common boundary. Note that this monotonicity is strict as a consequence of the boundary maximum principle for minimal surfaces. In the case of $\widetilde h_1$ and $\widetilde h_2$, $\nu_\lambda$ additionally does not change sign due to item (b).
\end{proof}

\subsection{The conjugate $H$-immersion}\label{subsec:conjugate}
Let $\Sigma_\lambda \subset \mathbb{M}^2(\kappa) \times \mathbb{R}$ be the conjugate of the surface $\widetilde{\Sigma}_\lambda$ defined in Section~\ref{subsec:plateau}. Therefore, $\Sigma_\lambda$ is a compact $H$-surface whose boundary $\Gamma_\lambda$ consists of three curves $h_0$, $h_1$ and $h_2$ contained in vertical planes $P_0$, $P_1$ and $P_2$, respectively, and a curve $v$ lying in a slice, which will be assumed to be $\mathbb{M}^2(\kappa)\times \{0\}$ after a vertical translation (see Section~\ref{subsec:conjugate-Plateau-technique} and Figure~\ref{fig:conjugate-polygon}). Note that $\Sigma_\lambda$ has angles of $\frac\pi2$ at its vertexes $1$, $2$, $3$ and $4$, so it becomes a complete $H$-surface $\Sigma_\lambda^*$ by successive mirror symmetries about $P_0$, $P_1$, $P_2$ and $\mathbb{M}^2(\kappa)\times \{0\}$. The case $\lambda \in [0, \frac{\pi}{2}]$ was described in~\cite[Theorem 1]{MT} so we will assume that $\lambda > \frac{\pi}{2}$ in the sequel.

\begin{figure}[htbp]
  \centering
  \includegraphics{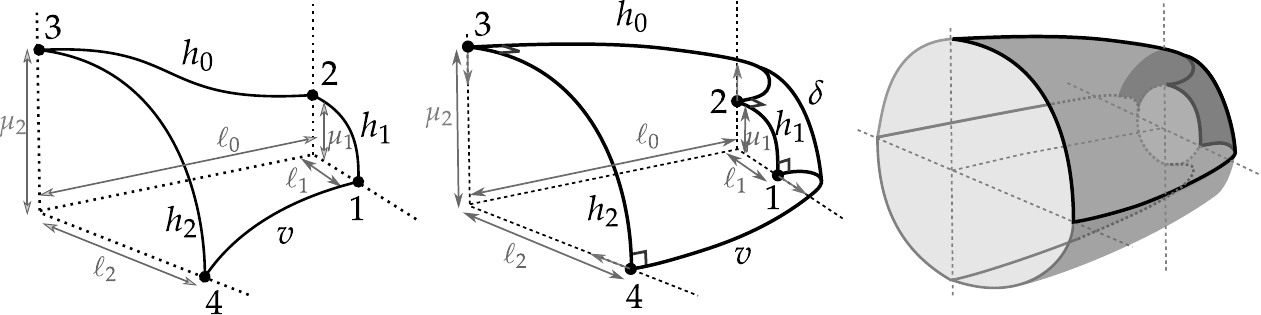}
  \caption{Conjugate polygon $\Gamma_\lambda$ for $\lambda < \frac{\pi}{2}$ (left), $\lambda > \frac{\pi}{2}$ (center) and the fundamental annulus $A_\lambda$ for $\lambda > \frac{\pi}{2}$ (right) obtained by reflecting $\Sigma_\lambda$ in the vertical plane containing $h_0$ and the slice containing $v$.}
  \label{fig:conjugate-polygon}
\end{figure}

\begin{remark}\label{rmk:round-sphere2}
If $\kappa=0$, then Remark~\ref{rmk:round-sphere1} ensures that $\widetilde\Sigma_\lambda$ is equivariant. The uniqueness in Lawson correspondence (up to ambient isometries) implies that $\Sigma_\lambda^*\subset\mathbb{R}^3$ is also equivariant. Due to the above geometric depiction of $\Sigma_\lambda^*$, along with the fact that it stays at bounded distance from the straight line $\Gamma=P_0\cap(\mathbb{R}^2\times\{0\})$, we easily infer that $\Sigma_\lambda^*$ is one of the classical Delaunay $H$-surfaces in $\mathbb{R}^3$, and it is rotationally invariant about $\Gamma$.
\end{remark}

For each $i\in\{0,1,2\}$, we can express $h_i=(\beta_i,z_i)\in\mathbb{M}^2(\kappa)\times\mathbb R$, and it follows that $\|\beta_i'\|=|\nu_\lambda|$ and $|z_i'|=(1-\nu_\lambda^2)^{1/2}$ because $\Sigma_\lambda^*$ intersects $P_i$ orthogonally, see~\cite[Section~5.2]{MT}. In view of Proposition~\ref{prop:angle}, we deduce that $\beta_1$, $\beta_2$, $z_0$, $z_1$ and $z_2$ are injective, and $\beta_0$ can be split into two injective subcurves by cutting at the point where the angle function $\nu_\lambda$ changes sign. Since the vertical planes $P_1$ and $P_2$ are orthogonal to $P_0$, it follows that $\Sigma_\lambda^*$ is invariant under horizontal translations of length $2\ell_0(\lambda)$, where 
\begin{align}\label{eqn:definition-ell-mu}
\ell_i(\lambda)&=-\int_{h_i}\nu_\lambda,&
\mu_i(\lambda)&=\int_{h_i}\sqrt{1-\nu_\lambda^2}
\end{align}
denote, respectively, the (signed) length of the projection of $h_i$ to $\mathbb{M}^2(\kappa)$ and the difference of heights of the endpoints of $h_i$, for $i\in\{0,1,2\}$, see Figure~\ref{fig:conjugate-polygon}. 

\begin{corollary}\label{cor:monotonicity-ell}
The functions $\lambda\mapsto\ell_i(\lambda)$ satisfy the following monotonicity properties:
 \begin{enumerate}
   \item[(a)] $\lambda\mapsto\ell_0(\lambda)$ is strictly decreasing and positive on $[0,+\infty).$
   \item[(b)] $\lambda\mapsto\ell_1(\lambda)$ is strictly decreasing on $[0,+\infty)$ with $\ell_1(\frac{\pi}{2})=0$.
   \item[(c)] $\lambda\mapsto\ell_2(\lambda)$ is strictly increasing  and positive on $[0,+\infty)$.
 \end{enumerate}
 
 \end{corollary} 

 \begin{proof}
 Using the monotonicity of the angle function in Proposition~\ref{prop:angle} for $\lambda\geq\frac{\pi}{2}$ and in~\cite{MT} for $0\leq\lambda\leq\frac\pi2$, we get that, if $0\leq\lambda_1<\lambda_2$, then
 \[\ell_0(\lambda_1) = -\int_{h_0} \nu_{\lambda_1} >-\int_{h_0} \nu_{\lambda_2} = \ell_0(\lambda_2).\]
 The helicoid $S$, as a barrier for the solution of the Plateau problem, lies above $\widetilde\Sigma_\lambda$ in a neighbourhood of $\widetilde h_0$ in the model $M(4H^2+\kappa,H)$, see Figure~\ref{fig:polygon-Berger}. Since both $\nu_\lambda$ and $\nu_S$, the angle function of the helicoid, do not take the values $\pm 1$ in the interior of $\widetilde h_0$, one easily infers that $-1<\nu_\lambda<\nu_S<1$ along $\widetilde h_0$ for all $\lambda>0$. Since the helicoid is symmetric with respect to its axis, we deduce that $\int_{\widetilde h_0}\nu_S=0$ and hence $\ell_0(\lambda)=\int_{\widetilde h_0}(-\nu_\lambda)>0$ for all $\lambda>0$.

 We will finish by discussing only item (b), because item (c) follows from similar arguments. Proposition~\ref{prop:angle} and~\cite{MT} again yield the estimate
 \[\ell_1(\lambda_1)= -\int_{h_1([\frac{\pi}{4}, \frac{\lambda_1}{2}])} \nu_{\lambda_1} > -\int_{h_1([\frac{\pi}{4}, \frac{\lambda_1}{2}])}\nu_{\lambda_2} > -\int_{h_1([\frac{\pi}{4}, \frac{\lambda_2}{2}])} \nu_{\lambda_2} = \ell_1(\lambda_2).\]
 This first inequality follows by distinguishing cases depending on whether $\lambda_1$ and $\lambda_2$ lie in $[0,\frac{\pi}{2}]$ or in $(\frac\pi2,+\infty)$, plus the fact that $\lambda\mapsto\nu_\lambda$ is positive and increasing along $h_1$ if $\lambda\geq\frac\pi2$, and negative and increasing if $0\leq\lambda\leq\frac\pi2$. The second inequality follows from enlarging $[\frac{\pi}{4}, \frac{\lambda_1}{2}]$ to $[\frac{\pi}{4}, \frac{\lambda_2}{2}]$ taking the signs into account. Notice that $\ell_1(\frac{\pi}{2})=0$ because $h_1$ reduces to a point for $\lambda=\frac\pi2$.
 \end{proof}

\begin{remark}[Compactness]\label{rmk:compactness}
Assuming that $\kappa>0$, the surface $\Sigma_\lambda^*$ is compact if and only if $\ell_0(\lambda)$ is a rational multiple of $\frac{2\pi}{\sqrt\kappa}$, the length of a great circle of $\mathbb{S}^2(\kappa)$. Since $\ell_0(\lambda)$ is a positive continuous strictly decreasing function, we deduce that compact examples abound in the family $\Sigma_\lambda^*$ for $\lambda\geq0$. If the rationality condition does not hold, then $\Sigma_\lambda^*$ becomes dense in an open subset of $\mathbb{S}^2(\kappa)\times\mathbb{R}$.

If $\kappa\leq 0$, then item (a) of Corollary~\ref{cor:monotonicity-ell} evidences that $P_1$ and $P_2$ never coincide, whence $\Sigma_\lambda^*$ is a proper non-compact $H$-surface for all $\lambda\geq 0$.
\end{remark}

\section{The geometry of horizontal Delaunay surfaces}\label{sec:geometry}

This section is devoted to prove further properties of Delaunay surfaces, with special emphasis on embeddedness. We will develop a new approach that relies on finding a function in the kernel of the stability operator of conjugate surfaces that is produced simultaneously by two $1$-parameter groups of isometric deformations: the group $\{\Phi_t\}_{t\in\R}$ in $\mathbb{M}^2(\kappa) \times \mathbb{R}$ defined in Section~\ref{sec:foliation}, and the group $\{\widetilde\Phi_t\}_{t\in\R}$ in the Berger sphere $\mathbb{S}_b^3(4H^2+\kappa,H)$ given by $\widetilde\Phi_t(z,w) = (e^{-\frac{it}{2}}z, e^{\frac{2it}{2}}w)$.

\subsection{Uniqueness}
In the model $M(4H^2+\kappa,H)$ given by Equation~\eqref{eqn:metric-Ekt}, the aforesaid group $\{\widetilde\Phi_t\}_{t\in\R}$ corresponds to the screw-motions
\begin{equation}\label{eqn:screw-motions}
\widetilde\Phi_t(x,y,z)=\left(x\cos t+y\sin t,y\cos t-x\sin t,z+\tfrac{2H}{4H^2+\kappa}t\right),
\end{equation}
and is associated with the Killing vector field $\widetilde X=y\partial_x-x\partial_y+\tfrac{2H}{4H^2+\kappa}\partial_z$. This field has no zeros and gives rise to a Killing submersion $\Pi_0:M(\kappa,\tau)\to(\R^2,\df s^2)$ in the sense of~\cite{LM}, such that $\Pi_0(x,y,z)=(u,v)$ if and only if there exists $t\in\R$ such that $\widetilde\Phi_t(u,v)=(x,y,z)$; in particular, $\Pi_0(x,y,0)=(x,y)$ for all $(x,y)\in\R^2$. Note that the metric $\df s^2$ that makes $\Pi_0$ Riemannian has not constant curvature, the Killing vector field $\widetilde X$ has not constant length, and the bundle curvature is not constant. The horizontal geodesics $\widetilde h_1$ and $\widetilde h_2$ become vertical with respect to $\Pi_0$, whereas $\widetilde h_0$ and $\widetilde v$ are transversal to the fibers of $\Pi_0$. This means that $\widetilde\Gamma_\lambda$ is a \emph{Nitsche graph} with respect to $\Pi_0$ in the sense of~\cite[Definition~3.7]{Man12} for all $\lambda\geq 0$. As we justify next, the Nitsche condition implies the uniqueness of $\widetilde\Sigma_\lambda$ inside the mean convex body $\Omega$, which in turn implies that $\widetilde\Sigma_\lambda$ depends unambiguously and continuously on $\lambda\geq 0$. 

\begin{proposition}\label{prop:uniqueness}
Given $\lambda\geq0$, there exists a unique solution $\widetilde\Sigma_\lambda\subset\Omega$ of the Plateau problem with boundary $\widetilde\Gamma_\lambda$, and the interior of $\widetilde\Sigma_\lambda$ is a vertical graph with respect to $\Pi_0$.
\end{proposition}

\begin{proof}
We have already proved the existence of $\widetilde\Sigma_\lambda$ in Section~\ref{subsec:plateau}. Uniqueness follows from~\cite[Proposition~3.8]{Man12}, whose proof also works in the case the Killing vector field has bounded nonconstant length. As a consequence, the fact that the interior of $\widetilde\Sigma_\lambda$ is a graph can be deduced from the maximum principle by performing slight deformations of the boundary $\widetilde\Gamma_\lambda$ as in~\cite[Proposition~2]{MT}.
\end{proof}

\subsection{Stability of the fundamental annulus} 
To deal with the global geometry, we will drop the model $M(4H^2+\kappa,H)$ throughout the rest of the paper, and assume that $\widetilde\Sigma_\lambda$ is immersed in $\mathbb S^3_b(4H^2+\kappa,H)\subset\C^2$ via the local isometry $\Theta$ given by Equation~\eqref{eq:local-isometry-Daniel-Berger}. Define $\widetilde\Sigma^*_\lambda\subset\mathbb S^3_b(4H^2+\kappa,H)$ as the complete (immersed) minimal surface in $\mathbb{S}_b^3(4H^2,H)$ we obtain by extending $\widetilde\Sigma_\lambda$ across its boundary. The Killing field $\widetilde X$ is also globally expressed as $\widetilde X_{(z,w)}=\frac{i}{2}(-z,w)$. 

The fact that $\widetilde\Sigma_\lambda\subset\mathbb S^3_b(4H^2+\kappa,H)$ is transversal to $\widetilde X$ makes us consider the smooth function $u=\langle\widetilde X,\widetilde N\rangle$, which is positive in the interior of $\widetilde\Sigma_\lambda$ and vanishes along $\widetilde h_1$ and $\widetilde h_2$. Since $\widetilde X$ is Killing, the function $u$ lies in the kernel of the stability operator of $\widetilde\Sigma_\lambda$, given by
\begin{equation}\label{eqn:stability-operator}
L=\Delta-2K+4H^2+\kappa(1+\nu_\lambda^2).
\end{equation}
Recall that a closed domain $D$ of a complete Riemannian surface is called (strongly) \emph{stable} if the first eigenvalue of its stability operator is non-negative, i.e., if
\[\lambda_1(D)=\inf\left\{\frac{\int_{D}fLf}{\int_{D}f^2}:f\in C^\infty_0(D),f\not\equiv 0\right\}\geq 0,\]
where $C^\infty_0(D)$ denotes the set of compactly supported smooth functions on $D$. Observe that $\Sigma_\lambda^*$ cannot be stable as a whole for any $\lambda\geq 0$ because it is orientable and parabolic (it has linear area growth by an estimate similar to Lemma~\ref{lemma:lag}), and therefore its stability would contradict~\cite[Theorem~2]{MPR}.

Let $A_\lambda$ be the $H$-annulus in $\mathbb M^2(\kappa)\times\mathbb R$ that extends $\Sigma_\lambda$ by means of mirror symmetries across $P_0$ and $\mathbb{M}^2(\kappa)\times\{0\}$ (see Figure~\ref{fig:conjugate-polygon} right). It consists of four copies of $\Sigma_\lambda$ and will be called the \emph{fundamental annulus} of $\Sigma_\lambda^*$. Next proposition shows that $A_\lambda$ is a nodal set of the function $u$ and hence stable.

\begin{proposition}\label{prop:stability}
The annulus $A_\lambda$ is a maximal stable domain of $\Sigma_\lambda^*$ for all $\lambda>0$.
\end{proposition}

\begin{proof}
We will begin by showing that the smooth function $u=\langle\widetilde X,\widetilde N\rangle$ inherits the symmetries of $\widetilde\Sigma^*_\lambda$. If $R_\gamma$ denotes the axial symmetry about a horizontal or vertical geodesic containing a boundary component $\gamma\subset\widetilde\Gamma_\lambda$, then it is easy to check that
\begin{align*}
    R_{\widetilde{h}_0}(z, w) &= (\overline{z}, \overline{w}),&R_{\widetilde{v}}(z, w) &= (ie^{-i\lambda}w, -ie^{i\lambda}z),\\
    R_{\widetilde{h}_1}(z,w) &= (\overline{w}, \overline{z}),&R_{\widetilde{h}_2}(z,w) &= (-\overline{w}, -\overline{z}).
  \end{align*}
It turns out that $\widetilde\Phi_t\circ R_\gamma=R_\gamma\circ\widetilde\Phi_{-t}$ (and hence $(R_\gamma)_*\widetilde X=-\widetilde X$) if $\gamma$ is either $\widetilde h_0$ or $\widetilde v$; on the contrary, one has $\widetilde\Phi_t\circ R_\gamma=R_\gamma\circ\widetilde\Phi_{t}$ (and hence $(R_\gamma)_*\widetilde X=\widetilde X$) if $\gamma$ is $\widetilde h_1$ or $\widetilde h_2$. On the other hand, $(R_\gamma)_*\widetilde N=-\widetilde N$ for any of the four boundary components, where $\widetilde N$ is the extended unit normal to $\widetilde\Sigma_\lambda^*$. We deduce that $u$ is preserved by the symmetries about $\widetilde v$ or $\widetilde h_0$, and sent to $-u$ by the symmetries about $\widetilde h_1$ or $\widetilde h_2$ (note that $u=0$ along $\widetilde h_1$ and $\widetilde h_2$ because these curves are tangent to $\widetilde X$).

Observe that $u$ also produces a smooth function in the kernel of the stability operator of $\Sigma_\lambda^*$ because $\Sigma_\lambda^*$ and $\widetilde\Sigma_\lambda^*$ share the same stability operator~\eqref{eqn:stability-operator}, see~\cite[Proposition~5.12]{Daniel07}. Since axial symmetries in $\mathbb{S}_b^3(4H^2+\kappa,H)$ correspond to mirror symmetries in $\mathbb{M}^2(\kappa)\times\mathbb R$, it follows that the symmetries with respect to $P_0$ and $\mathbb M^2(\kappa)\times\{0\}$ preserve $u$, whereas the symmetries with respect to $P_1$ and $P_2$ send $u$ to $-u$. Proposition~\ref{prop:uniqueness} guarantees that $u>0$ on the interior of $\widetilde\Sigma_\lambda$, whence it remains positive in the interior of $A_\lambda$ by the aforesaid symmetries and vanishes identically along $\partial A_\lambda$. We deduce from classical elliptic theory that the first eigenvalue of its stability operator satisfies $\lambda_1(A_\lambda)=0$ and $\lambda_1(D)<0$ for any open domain $D\subset\Sigma_\lambda^*$ containing $A_\lambda$.
\end{proof}

If $X$ is the Killing field associated with the group $\{\Phi_t\}_{t\in\mathbb{R}}$ of translations along the axis $\Gamma=P_0\cap(\mathbb M^2(\kappa)\times\{0\})$, next corollary reveals that $\langle X,N\rangle$ is proportional to $\langle\widetilde X,\widetilde N\rangle$. Note that the constant of proportionality goes to zero as $\lambda\to 0$.

\begin{corollary}\label{cor:horizontal-graph}
If $\lambda>0$, the fundamental piece $\Sigma_\lambda$ is tangent to $X$ only on $h_1\cup h_2$.
\end{corollary}

\begin{proof}
  The function $w=\langle X,N\rangle$ belongs to the kernel of the stability operator $L$ of $A_\lambda$. Since $P_1$ and $P_2$ are orthogonal to $X$, we have $w=0$ along $h_1$ and $h_2$. Taking into account that $w$ lies in the eigenspace of $L$ associated with $0=\lambda_1(A_\lambda)$ and this subspace is $1$-dimensional, there exists $a\in\R$ (depending on $\lambda$) such that $w=au$. Observe that, if $w$ is identically zero, then $\Sigma_\lambda^*$ is invariant by $\{\Phi_t\}_{t\in\mathbb{R}}$, which only occurs when $\lambda=0$, but this case is excluded by assumption. Therefore, if $\lambda>0$, then $w$ is either positive or negative on the interior of $A_\lambda$, i.e., the interior of $A_\lambda$ is transversal to $X$. As the interiors of $h_0$ and $v$ lie in the interior of $A_\lambda$, we deduce that they are also transversal to $X$.
\end{proof}

\subsection{Embeddedness of unduloids} 
Just like in the vertical case, we infer from the description in Section~\ref{subsec:conjugate} that horizontal nodoids are not even Alexandrov-embedded for any $\lambda>\frac\pi 2$. On the contrary, we can say precisely when unduloids are embedded, which settles the question of embeddedness posed in~\cite{MT}.  

\begin{proposition}\label{prop:properly-embedded-unduloids}
If $\kappa\leq 0$, horizontal unduloids are properly embedded and non-compact.
\end{proposition}

\begin{proof}
We will fix $0<\lambda<\frac{\pi}{2}$. Properness and non-compactness of $\Sigma_\lambda$ are discussed in Remark~\ref{rmk:compactness}. As for embeddedness, we will begin by showing that each integral curve of $X$ intersects $A_\lambda$ at most in a point. Otherwise, consider the vertical plane $P\subset\mathbb{H}^2(\kappa)\times\mathbb\R$ containing two points of $A_\lambda$ in the same integral curve, so these two points lie at the same height with respect to $\mathbb{H}^2\times\{0\}$. From the fact that $\nu_\lambda$ does not vanish in the interior of $\Sigma_\lambda$, it is easy to realize that that height restricted to $P\cap\Sigma_\lambda$ must have an interior critical point (between the two points at the same height), so $X$ is tangent at such a critical point and we reach the desired contradiction. Thus $A_\lambda$ is an $H$-graph in the direction of $X$. Furthermore, the maximum principle with respect to minimal vertical planes, along with the boundary curvature estimates in~\cite{Manzano} (adapted to this periodic $H$-multigraph as in~\cite[Lemma~4.1]{MTAJM}), imply that $A_\lambda$ lies in the vertical slab demarcated by $P_1$ and $P_2$, and hence the complete surface $\Sigma_\lambda^*$ is also embedded.
\end{proof}

However, if $\kappa>0$, embeddedness finds an essential obstruction whenever $\Sigma_\lambda$ reaches the vertical geodesics $P_1\cap P_2$, i.e., if the projection of $\Sigma_\lambda$ to $\mathbb S^2(\kappa)$ runs over any of the poles defined by the great circle $\Gamma=P_0\cap(\mathbb{S}^2(\kappa)\times\{0\})$.

\begin{proof}[Proof of Theorem~\ref{thm:embeddedness}]
Assume that $\kappa>0$. Observe that $\lambda\mapsto\ell_0(\lambda)$ is positive and decreasing by Corollary~\ref{cor:monotonicity-ell}, so it ranges from $\ell_0(0)$ to $\ell_0(\frac{\pi}{2})$. On the one hand, $\ell_0(0)$ is the length of $\widetilde h_0$, a quarter of the length of a horizontal geodesic of $\mathbb{S}^3_b(4H^2+\kappa,H)$; on the other hand, $\ell_0(\frac{\pi}{2})$ is the radius of the domain (as a bigraph) of an $H$-sphere and can be computed from~\cite[p.\ 1268]{Man12} after rescaling the metric. This gives the estimate
\begin{equation}\label{thm:embeddedness:eqn1}
\frac{2}{\sqrt{ \kappa}}\arctan\frac{\sqrt{\kappa}}{2H}=\ell_0(\tfrac{\pi}{2})<\ell_0(\lambda)< \ell_0(0)=\frac{\pi}{\sqrt{4H^2+\kappa}}.
\end{equation}
For a fixed $H>0$, we are interested in values of $\lambda\in(0,\frac{\pi}{2})$ such that $\ell_0(\lambda)=\frac{\pi}{m\sqrt{\kappa}}$ for some $m\in\N$, i.e., such that $\Sigma_\lambda^*$ consists of $2m$ copies of $A_\lambda$ and closes its period in one turn around the axis $\Gamma$, for otherwise embeddedness fails (see also Remark~\ref{rmk:compactness}). Equation~\eqref{thm:embeddedness:eqn1} allows us to say that such compact $H$-unduloids are in correspondence with integers $m\geq 1$ satisfying
\begin{equation}\label{thm:embeddedness:eqn2}
   \frac{\sqrt{4H^2+\kappa}}{\sqrt{\kappa}}< m< \frac{\pi}{2\arctan(\frac{\sqrt{\kappa}}{2H})}.
\end{equation}
If $H\leq\frac{\sqrt{\kappa}}{2}$, no integer value of $m$ satisfies~\eqref{thm:embeddedness:eqn2}, but it is easy to realize that there actually exist such integer values of $m$ for all $H>\frac{\sqrt{\kappa}}{2}$ (see figure~\ref{fig:compact-embedded-moduli-space}). 

For a fixed integer $m\geq 2$, the inequality~\eqref{thm:embeddedness:eqn2} holds true if and only if $\frac{2H}{\sqrt{\kappa}}\in(\cot(\tfrac{\pi}{2m}),\sqrt{m^2-1})$. As $\lambda\mapsto\ell_0(\lambda)$ is continuous and strictly decreasing, there exists a unique value $\lambda=\lambda_m(H)$ such that $\ell_0(\lambda_m(H))=\frac{\pi}{m\sqrt{\kappa}}$ in the aforesaid range for $H$. This yields the existence of the family $\mathcal T_m$ in the statement, and the limit cases follow from the monotonicity of the family:
\begin{itemize}
  \item If $\frac{2H}{\sqrt{\kappa}}=\cot(\tfrac{\pi}{2m})$, then $m=\frac{\pi}{2\arctan(\frac{\sqrt{\kappa}}{2H})}$, and hence $\ell_0(\frac\pi2)=\ell_0(\lambda)$. This means that $\lambda=\frac\pi2$ and the surface reduces to a stack of $m$ tangent $H$-spheres.
  \item Likewise, if $\frac{2H}{\sqrt{\kappa}}=\sqrt{m^2-1}$, then $\lambda=0$, and the surface is an $H$-cylinder.
\end{itemize}
It remains to prove that all these examples are embedded. On the one hand, observe that $\ell_2(\frac\pi2)$ is the radius of the circle of $\mathbb{S}^2(4H^2+\kappa)$ over which the $H$-sphere $\Sigma_{\pi/2}^*$ is a bigraph. This radius is at most a quarter of the length of a great circle of $\mathbb{S}^2(4H^2+\kappa)$ if $H>\frac{\sqrt{\kappa}}{2}$. Using the fact that $0<\ell_2(\lambda)<\ell_2(\frac{\pi}{2})$ (see Corollary~\ref{cor:monotonicity-ell}), we deduce that $h_2$ does not reach $P_1\cap P_2$. On the other hand, again by Corollary~\ref{cor:monotonicity-ell}, we have $\ell_1(\lambda)\leq\ell_1(0)=\ell_2(0)\leq\ell_2(\lambda)$, where we have used that $\Sigma_0^*$ is invariant by $\{\Phi_t\}_{t\in\mathbb{R}}$, whence $h_1$ does not intersect $P_1\cap P_2$ either. This implies that $\Sigma_\lambda\cap P_1\cap P_2=\emptyset$ because otherwise the annulus $A_\lambda$ would have an interior point lying in $P_1\cap P_2$; since $X$ identically vanishes on $P_1\cap P_2$, this would contradict Corollary~\ref{cor:horizontal-graph}. Once we have ensured $\Sigma_\lambda^*$ is away from $P_1\cap P_2$, the same argument as in the proof of Proposition~\ref{prop:properly-embedded-unduloids} ensures that $A_\lambda$ is embedded and lies in the wedge between $P_1$ and $P_2$, so we are done.
\end{proof}

\subsection{Maximum height}\label{subsec:monotonicity}

Since $\Sigma_\lambda^*$ is periodic in a horizontal direction, we can ensure the existence of a point with maximum height over the horizontal plane of symmetry $\mathbb{M}^2(\kappa)\times\{0\}$. This point must be the vertex $3$ in view of Proposition~\ref{prop:angle}, and hence the maximum height is $\mu_2(\lambda)$, see Figure~\ref{fig:conjugate-polygon}.

\begin{proposition}\label{prop:height}
The maximum height of the horizontal Delaunay surface $\Sigma_\lambda^*$ is strictly increasing as a function of $\lambda$.
\end{proposition}

\begin{proof}
We will prove that $\mu_2(\lambda_1)<\mu_2(\lambda_2)$ whenever $0\leq\lambda_1<\lambda_2$. Using the $1$-parameter group of screw-motions $\{\widetilde\Phi_t\}$, let $\widetilde\Sigma_{\lambda_1}'=\widetilde\Phi_{\lambda_2-\lambda_1}(\widetilde\Sigma_{\lambda_1})$. Then $\widetilde\Sigma_{\lambda_1}'$ and $\widetilde\Sigma_{\lambda_2}$ are Killing graphs and their boundaries are Nitsche graphs over the same domain of $\R^2$ for the Killing submersion $\Pi_0$. Moreover, these boundaries are ordered as Nitsche graphs in the sense of~\cite[Proposition~3.8]{Man12}, so $\widetilde\Sigma_{\lambda_1}'$ is located above $\widetilde\Sigma_{\lambda_2}$ in the model $M(4H^2+\kappa,H)$ (alternatively, we could argue that $\widetilde\Sigma_{\lambda_1}'$ acts as a barrier in the solution of the Plateau problem for $\widetilde\Gamma_{\lambda_2}$). 

This enables a comparison of the angle functions of $\widetilde\Sigma_{\lambda_1}'$ and $\widetilde\Sigma_{\lambda_2}$ along their common boundary. Equivalently, the angle functions of $\widetilde\Sigma_{\lambda_1}$ and $\widetilde\Sigma_{\lambda_2}$ are comparable through $\widetilde\Phi_{\lambda_2-\lambda_1}$ (note that this isometry preserves the angle function), and we get that $-1<\nu_{\lambda_1}<\nu_{\lambda_2}\circ\widetilde\Phi_{\lambda_2-\lambda_1}<0$, and hence
\begin{equation}\label{prop:height:eqn1}
\textstyle\sqrt{1-\nu_{\lambda_1}^2}<\sqrt{1-(\nu_{\lambda_2}\circ\widetilde\Phi_{\lambda_2-\lambda_1})^2}
\end{equation}
on $h_2((\frac{-\pi}{4},\frac{\lambda_1}{2}))$. Integrating~\eqref{prop:height:eqn1} along this curve and then enlarging the interval to $(\frac{-\pi}{4},\frac{\lambda_2}{2})$ in the same fashion as the proof of Corollary~\ref{cor:monotonicity-ell}, we deduce that the maximum heights satisfy $\mu_2(\lambda_1)<\mu_2(\lambda_2)$.
\end{proof}

\begin{remark}
The very same argument as in the proof of Proposition~\ref{prop:height} shows that $\mu_1(\lambda)$ is also a strictly increasing function of $\lambda$.
\end{remark}

\end{document}